\documentclass[leqno]{siamltex}
\usepackage{amssymb}
\usepackage{amsbsy}
\usepackage{newlfont}
\usepackage{amsmath}
\usepackage[normalem]{ulem}
\usepackage{psfrag}
\usepackage{graphicx}
\usepackage{color}  
\usepackage{xspace,float,stackrel,enumerate,subfigure}
\usepackage{bm}
\usepackage{bbm}
\usepackage{mathrsfs} 
\usepackage{remarks}
\usepackage{mydef}
\usepackage[numbers]{natbib}

\usepackage{adjustbox}

\usepackage{tikz}
\usepackage{pgfplots}
\usepackage{dsfont}
\usepackage{caption}
\let\cite=\citet

\usepackage{algorithm,algorithmic}
\usepackage{color} 
\usepackage{etoolbox}
\usepackage{marginnote}
\makeatletter
\patchcmd{\@mn@margintest}{\@tempswafalse}{\@tempswatrue}{}{}
\patchcmd{\@mn@margintest}{\@tempswafalse}{\@tempswatrue}{}{}
\makeatother
\newcounter{mnote}
\begin{document}
\newcommand\footnotemarkfromtitle[1]{%
\renewcommand{\thefootnote}{\fnsymbol{footnote}}%
\footnotemark[#1]%
\renewcommand{\thefootnote}{\arabic{footnote}}}

\title{Invariant domains preserving ALE \\ approximation of hyperbolic
  systems with \\ continuous finite elements
  \footnotemark[1]}

\author{Jean-Luc Guermond\footnotemark[2]
\and Bojan Popov\footnotemark[2]
\and Laura Saavedra\footnotemark[3]
\and Yong Yang\footnotemark[2] }
\date{Draft version \today}

\maketitle

\renewcommand{\thefootnote}{\fnsymbol{footnote}} \footnotetext[1]{
  This material is based upon work supported in part by the National
  Science Foundation grants DMS-1217262, by the Air
  Force Office of Scientific Research, USAF, under grant/contract
  number FA99550-12-0358, and by the Army Research Office under grant/contract
  number W911NF-15-1-0517.  Draft
  version, \today} \footnotetext[2]{Department of Mathematics, Texas
  A\&M University 3368 TAMU, College Station, TX 77843, USA.}
\footnotetext[3]{Departamento Fundamentos Matem\'aticos, Universidad Polit\'ecnica de Madrid, E.T.S.I. Aeron\'auticos, 28040 Madrid, Spain}
\renewcommand{\thefootnote}{\arabic{footnote}}

\begin{abstract}
  A conservative invariant domain preserving Arbitrary Lagrangian
  Eulerian method for solving nonlinear hyperbolic systems is
  introduced.  The method is explicit in time, works with continuous
  finite elements and is first-order accurate in space. One
  originality of the present work is that the artificial viscosity is
  unambiguously defined irrespective of the mesh
  geometry/anisotropy and does not depend on any ad hoc parameter.  The
  proposed method is meant to be a stepping stone for the construction
  of higher-order methods in space by using appropriate limitation
  techniques.
\end{abstract}



\pagestyle{myheadings} \thispagestyle{plain} 
\markboth{J.-L. GUERMOND,  B. POPOV, L. SAAVEDRA, Y. YANG }{Invariant domain preserving ALE}

\section{Introduction} \label{Sec:introduction} 
Consider the following
hyperbolic system in conservative form
\begin{equation}
 \label{def:hyperbolic_system} 
 \begin{cases} \partial_t \bu + \DIV \bef(\bu)=0, 
\quad \mbox{for}\, (\bx,t)\in \Real^d\CROSS\Real_+.\\
\bu(\bx,0) = \bu_0(\bx), \quad \mbox{for}\, \bx\in \Real^d,
\end{cases}
\end{equation}
where the dependent variable $\bu$ is $\Real^m$-valued and the flux
$\bef$ is $\Real^{m\CROSS d}$-valued.  The objective of this paper is
to investigate an approximation technique for solving
\eqref{def:hyperbolic_system} using an Arbitrary Lagrangian Eulerian
(ALE) formulation with continuous finite elements and explicit time
stepping on non-uniform meshes in any space dimension.

The interest for finite elements in the context of compressible
Lagrangian hydrodynamics has been recently revived by the work of
\cite{Dobrev12}, where the authors have demonstrated that high-order
finite elements have good properties in terms of geometry
representation, symmetry preservation, resolution of shock fronts and
high-order convergence rate for smooth solutions. The finite element
formalism has been combined with staggered grid hydrodynamics methods
in \cite{Barlow07,Scovazzi08} and with cell-centered hydrodynamics
methods in \cite{Vilar14} in the form of a Discontinuous Galerkin
scheme. One common factor of all these papers is that the
stabilization is done by introducing some artificial viscosity to
control post-shock oscillations, and high-order convergence in space
is achieved by restricting the diffusion to be active only in the
singular regions of the solution. This can be done in many ways, for
instance by measuring smoothness like in the ENO/WENO literature like
in \cite{Cheng_Shu_2014} or by using entropies like in the entropy
viscosity methodology of \cite{Guermond_Popov_Tomov_2016}. A detailed
list of requirements and specific artificial viscosity expressions for
Lagrangian hydrodynamics have been proposed by
\cite{CaramanaWhalen98}, \cite{Caramana_Loubere_JCP_2006},
\cite{Campbell01, Kolev09}, \cite{Lipnikov_Shashkov_2010}. The
artificial viscosity can also be implicitly introduced by using
Godunov type or cell-centered type methods based on Riemann solvers,
see \eg \cite{Bocheri_Dumbser_2014,Carre_Despres_2009}.

In the present paper we revisit the artificial viscosity problem for
general hyperbolic systems like \eqref{def:hyperbolic_system} using an
Arbitrary Lagrangian Eulerian (ALE) formulation and explicit
  time stepping. The originality of the present work is that (i) The
  approximation in space is done with continuous finite elements of
  arbitrary order; (ii) The local shape functions can be linear
  Lagrange elements, Bernstein-Bezier elements of any order or any
  other nonnegative functions (not necessarily polynomials) that have
  the partition of unity property; (iii) The finite element meshes
  considered are non-uniform, curvilinear and the space dimension is
  arbitrary; (iv) The artificial viscosity is unambiguously defined
  irrespective of the mesh geometry/anisotropy, does not depend on any
  ad hoc parameter, contrary to what has been previously done in the
  finite element literature, and leads to precise invariant domain
  properties and entropy inequalities; (v) The methods works for {\em
    any} (reasonable in the sense of \S\ref{Sec:Riemann}) hyperbolic
  system.  Although entirely based on continuous finite elements, our
work is deeply rooted in the work of \cite{Lax54} and
\cite{Hoff_1985}, and in some sense mimics well established
  schemes from the finite volume literature,
  \eg~\cite{Guillard_Farhat_2000,Farhat_Geuzaine_Grandmont_2001} or
  the discontinuous Galerkin literature, \eg \cite{Vilar14}.  The
proposed method is meant to be a stepping stone for the construction
of higher-order methods in space by using, for instance, the flux
transport correction methodology \`a la Boris-Book-Zalesak, or any
generalization thereof, to implement limitation. None of these
  generalizations are discussed in the paper. The sole objective of
  the present work is to give a firm, undisputable, theoretical
  footing to the first-order, invariant-domain preserving, method.

The paper is organized as follows. We introduce some notation and
recall important properties about the one-dimensional Riemann problem
in \S\ref{Sec:Riemann}. We introduce notation relative to mesh motion
and Lagrangian mappings in \S\ref{Sec:Preliminaries}.  The results
established in \S\ref{Sec:Riemann} and \S\ref{Sec:Preliminaries} are
standard and will be invoked in \S\ref{Sec:ALE_scheme} and
\S\ref{Sec:stability_analysis}.  The reader who is familiar with these
notions is invited to go directly to \S\ref{Sec:ALE_scheme}. We
describe in \S\ref{Sec:ALE_scheme} two versions of an ALE algorithm to
approximate the solution of \eqref{def:hyperbolic_system}. The first
algorithm, henceforth referred to as version~1, is composed of the
steps
\eqref{alg1_motion_of_ai}-\eqref{alg1_def_of_fmi}-\eqref{def_of_scheme_dij}-\eqref{def_of_dij}.
The second algorithm, \ie version~2, is composed of the steps
\eqref{motion_of_ai_tln}-\eqref{def_of_cij_gauss}-\eqref{def_mass_gauss}-\eqref{def_of_scheme_dij_gauss}-\eqref{def_of_dij}. The
key difference between these two algorithms is the way the mass
carried by the shape functions is updated (compare
\eqref{alg1_def_of_fmi} and \eqref{def_mass_gauss}).  Only version~1
can be easily made high-order in time by means of the strong stability
preserving technology (see
\cite{Ferracina_Spijker_2005,Higueras_2005,Gottlieb_Ketcheson_Shu_2009}
for details on SSP techniques).  It is proved in
\S\ref{Sec:stability_analysis} that under the appropriate CFL
condition both algorithms are invariant domain preserving,
conservative and satisfy a local entropy inequality for any admissible
entropy pair. The main results of this section are
Theorem~\ref{Thm:invariant_domain} and
Theorem~\ref{Thm:disrete_entropy_inequality}.  The SSP RK3 extension
of scheme~1 is tested numerically in \S\ref{Sec:numerical_tests} on
scalar conservation equations and on the compressible Euler equations
using two different finite element implementations of the method.  In
all the cases the ALE velocity is ad hoc and no particular effort has
been made to optimize this quantity. The purpose of this paper is not
to design an optimal ALE velocity but to propose an algorithm that is
conservative and invariant domain preserving for {\em any} reasonable
ALE velocity.

\section{Riemann problem and invariant domain} \label{Sec:Riemann}
We recall in this section elementary properties of Riemann problems that will be used
in the paper.
\subsection{Notation and boundary conditions}
In this paper the dependent variable $\bu$ in
\eqref{def:hyperbolic_system} is considered as a column vector
$\bu=(u_1,\ldots,u_m)\tr$. The flux is a matrix with entries
$f_{ij}(\bu)$, $1\le i\le m$, $1\le j\le d$. We denote $\bef_i$ the
row vector $(f_{i1},\ldots,f_{id})$, $i\in\intset{1}{m}$.  We denote
by $\DIV\bef$ the column vector with entries $(\DIV\bef)_i= \sum_{1\le
  j\le d}\partial_{x_j} f_{ij}$. For any $\bn=(n_1\ldots,n_d)\tr\in
\Real^d$, we denote $\bef(\bu)\SCAL\bn$ the column vector with entries
$\bef_i(\bu)\SCAL \bn =\sum_{1\le l\le d} n_lf_{il}(\bu)$, where
$i\in\intset{1}{m}$. Given two vector fields, say $\bu\in \Real^m$ and
$\bv\in \Real^d$, we define $\bu\otimes \bv$ to be the $m\CROSS d$
matrix with entries $u_iv_j$, $i\in\intset{1}{m}$,
$j\in\intset{1}{d}$. We also define $\DIV(\bu\otimes \bv)$ to be the
column vector with entries $\DIV(\bu\otimes \bv)_i=
\sum_{j=1}^d\partial_j (u_i v_j)$.  The unit sphere in $\Real^d$
centered at $0$ is denoted by $S^{d-1}(\bzero,1)$.

To simplify questions regarding boundary conditions, we assume that
the initial data is constant outside a compact set and we solve the
Cauchy problem in $\Real^d$ or we use periodic boundary conditions.

\subsection{One-dimensional Riemann problem}
We are not going to try to define weak solutions to
\eqref{def:hyperbolic_system}, but instead we assume that there is a
clear notion for the solution of the Riemann problem. 
To stay general we introduce a generic hyperbolic flux $\bh$ and we say that $(\eta,\bq)$ is an entropy pair associated
with the flux $\bh$ if $\eta$ is convex and the following identity holds:
\begin{equation}
\partial_{v_k} (\bq(\bv)\SCAL \bn) = \sum_{i=1}^m \partial_{v_i}\eta(\bv) 
\partial_{v_k}(\bh_i(\bv)\SCAL \bn), \qquad \forall k\in\intset{1}{m},\ 
\forall \bn\in S^{d-1}(\bzero,1). \label{def_entropy}
\end{equation}
We refer to \cite[\S2]{Chen_2005} for more details on convex entropies
and symmetrization.  In the rest of the paper we assume that there
exists a nonempty admissible set $\calA_\bh\subset\Real^{m}$ such that
the following one-dimensional Riemann problem
\begin{equation}
 \label{def:Riemann_problem} 
  \partial_t \bu + \partial_x (\bh(\bu)\SCAL\bn)=0, 
\quad  (x,t)\in \Real\CROSS\Real_+,\qquad 
\bu(x,0) = \begin{cases} \bu_L, & \text{if $x<0$} \\ \bu_R,  & \text{if $x>0$}, \end{cases}
\end{equation}
has a unique entropy satisfying solution for any pair of states
$(\bu_L,\bu_R)\in \calA_\bh\CROSS \calA_\bh$ and any unit vector $\bn\in
S^{d-1}(\bzero,1)$. We henceforth denote the solution to this problem
by $\bu(\bh,\bn,\bu_L,\bu_R)$. We also say that $\bu$ is an entropy satisfying solution of
\eqref{def:Riemann_problem} if the following holds in the distribution
sense
\begin{equation}
\partial_t \eta(\bu) + \partial_x (\bq(\bu)\SCAL\bn) \le 0.
\end{equation}
for any entropy pair $(\eta,\bq)$.

It is unrealistic to expect a general theory of the Riemann problem~\eqref{def:Riemann_problem} 
for arbitrary nonlinear hyperbolic systems with large data, we
henceforth make the following assumption:
\begin{align}
&\begin{aligned}
&\text{The unique solution of \eqref{def:Riemann_problem} has a finite speed of} \\[-3pt]
&\text{propagation for any $\bn$ and any $(\bu_L,\bu_R)\in \calA_\bh\CROSS\calA_\bh$, \ie}\\[-3pt]
&\text{there are $\lambda_{L}(\bh,\bn,\bu_L,\bu_R)\le \lambda_{R}(\bh,\bn,\bu_L,\bu_R)$ such}\\
&\bu(x,t)= \begin{cases}
\bu_L, & \text{if $x \le t \lambda_{L}(\bh,\bn,\bu_L,\bu_R)$}\\
\bu_R, & \text{if $x \ge   t \lambda_{R}(\bh,\bn,\bu_L,\bu_R)$}.
\end{cases} 
\end{aligned} \label{finite_speed}
\end{align}
This assumption is known to hold for small data when the system is
strictly hyperbolic with smooth flux and all the characteristic fields
are either genuinely nonlinear or linearly degenerate. More precisely
there exists $\delta>0$ such that the Riemann problem has a unique
self-similar weak solution in Lax's form for any initial data such
that $\|\bu_L-\bu_R\|_{\ell^2}\le \delta$, see \cite{Lax_1957_II} and
\cite[Thm~5.3]{Bressan_2000}. In particular there are $2m$ numbers
$\lambda_1^-\le \lambda_1^+ \le  \lambda_2^-\le \lambda_2^+ \le \ldots \le 
\lambda_m^-\le \lambda_m^+$,
defining up to $2m+1$ sectors (some could be empty) in the $(x,t)$ plane:
\begin{equation}
\frac{x}{t}\in (-\infty,\lambda_1^-), \quad
\frac{x}{t}\in (\lambda_1^-,\lambda_1^+), \ldots, \quad
\frac{x}{t}\in (\lambda_m^-,\lambda_m^+), \quad \frac{x}{t}\in (\lambda_m^+,\infty). 
\end{equation}
where the Riemann solution is $\bu_L$ in the sector $\frac{x}{t}\in
(-\infty,\lambda_1^-)$, $\bu_R$ in the last sector $\frac{x}{t}\in
(\lambda_m^+,\infty)$, and either a constant state or an expansion in
the other sectors, see \citep[Chap.~5]{Bressan_2000}.  
In this case we have $\lambda_{L}:=\lambda_{L}(\bh,\bn,\bu_L,\bu_R)=\lambda_1^-$
and $\lambda_{R}:=\lambda_{R}(\bh,\bn,\bu_L,\bu_R)=\lambda_m^+$. The sector
$\lambda_1^- t < x < \lambda_m^+ t$, $0<t$, is henceforth referred to
as the Riemann fan.  The maximum wave speed in the Riemann fan is
$\lambda_{\max}:=\lambda_{\max}(\bh,\bn,\bu_L,\bu_R):= \max(|\lambda_{L}|,
|\lambda_{R}|)$. For brevity, when there is no ambiguity, we will omit the dependence of 
$\lambda_{L}, \lambda_{R}$ and $\lambda_{\max}$ on the parameters $\bh,\bn,\bu_L,\bu_R$.
The finite speed assumption \eqref{finite_speed} holds in the case of strictly
hyperbolic systems that may have characteristic families that are
either not genuinely nonlinear or not linearly degenerate, see \eg
\cite[Thm.~9.5.1]{Dafermos_2000}.  
%
\subsection{Invariant sets and domains}
The following elementary result is an important, well-known,
consequence of the Riemann fan assumption~\eqref{finite_speed}:
\begin{lemma} \label{Lem:elementary_Riemann_pb} Let 
$\bh$ be a hyperbolic flux over the admissible set $\calA_\bh$ and
satisfying the finite wave speed assumption \eqref{finite_speed}.
Let $\bv(\bh,\bn,\bv_L,\bv_R)$ be the  unique solution to the problem
  $\partial_t \bv +\partial_x(\bh(\bv)\SCAL\bn)=0$
with initial data $\bv_L,\bv_R \in
  \calA_\bh$. Let $(\eta,\bq)$ be an entropy pair associated with the flux
  $\bh$. Assume that $t\,
  \lambda_{\max}(\bh,\bn,\bv_L,\bv_R) \le \frac12$
and let $ \overline\bv(t,\bh,\bn,\bv_L,\bv_R) :=\int_{-\frac12}^{\frac12}
  \bv(\bh,\bn,\bv_L,\bv_R)(x,t) \diff \bx $, then
\begin{align}
  \overline\bv(t,\bh,\bn,\bv_L,\bv_R)  &= \tfrac{1}{2}(\bv_L+\bv_R) 
  - t\big(\bh(\bv_R)\SCAL\bn - \bh(\bv_L)\SCAL\bn\big).
\label{elementary_Riemann_pb}\\
  \eta(\overline\bv(t,\bh,\bn,\bv_L,\bv_R) )
&\le \tfrac{1}{2}(\eta(\bv_L)+\eta(\bv_R)) 
  - t(\bq(\bv_R)\SCAL\bn - \bq(\bv_L)\SCAL\bn).
\label{entropy_elementary_Riemann_pb}
\end{align}
\end{lemma}%

We now introduce the notions of invariant sets. The definitions that
we adopt are slightly different from what is usually done in the
literature (see \eg in
\cite{Chueh_Conley_Smoller,Hoff_1985,Frid_2001}.
\begin{definition}[Invariant set] \label{Def:invariant_set} Let $\bh$
  be a hyperbolic flux over the admissible set $\calA_\bh$ and
  satisfying the finite wave speed assumption \eqref{finite_speed}.
  We say that a convex set $A\subset \calA_\bh\subset \Real^m$ is
  invariant for the problem $\partial_t \bv +\DIV \bh(\bv)=0$ if for
  any pair $(\bv_L,\bv_R)\in A\CROSS A$, any unit vector $\bn\in
  \calS^{d-1}(\bzero,1)$, the average of the entropy solution of the
  Riemann problem $\partial_t \bv +\DIV (\bh(\bv)\SCAL \bn)=0$ over
  the Riemann fan $\frac{1}{t(\lambda_{R}-\lambda_{L})}
  \int_{\lambda_{L}t}^{\lambda_{R}t} \bv(\bh,\bn,\bv_L,\bv_R)(x,t)
  \diff \bx$, remains in $A$ for all $t>0$.
\end{definition}

\begin{remark} \label{Rem:Riemann_fan_average}
The above definition implies that
$\frac{1}{I} \int_{I} \bv(\bh,\bn,\bv_L,\bv_R)(x,t) \diff \bx\in A$ for
any $t>0$ and any interval $I$ such that
$(\lambda_{L}t,\lambda_{R}t)\subset I$.
\end{remark}

\begin{lemma}[Translation] \label{Lem:same_invariant_sets}
  Let $\bsfW\in\Real^d$ and let
  $\bg(\bv) := \bef(\bv) - \bv\otimes\bsfW$.
\begin{enumerate}[(i)]
\item \label{Lem:same_invariant_sets:item1} The two problems:
  $\partial_t \bu + \DIV\bef(\bu)=0$ and $\partial_t \bv +
  \DIV\bg(\bv)=0$ have the same admissible sets and the same invariant
  sets.
\item \label{Lem:same_invariant_sets:item2} $(\eta(\bu),\bq(\bu))$ is
  an entropy pair for the flux $\bef$ if and only if
  $(\eta(\bv),\bq(\bv)-\eta(\bv)\bsfW)$ is an entropy pair for the
  flux $\bg$.
\end{enumerate}
\end{lemma}
\begin{proof}
  \eqref{Lem:same_invariant_sets:item1} Given $\bn\in
  \calS^{d-1}(\bzero,1)$, the solutions of the Riemann problems
  $\partial_t \bu + \partial_x(\bef(\bu)\SCAL\bn)=0$ and $\partial_t
  \bv + \partial_x(\bg(\bv)\SCAL\bn)=0$, with the same initial data,
  are such that $\bv(x,t) = \bu(x+(\bsfW\SCAL\bn) t,t)$. Therefore the
  admissible sets and the invariant sets are identical.
  \eqref{Lem:same_invariant_sets:item2}
Let $\bn\in
  \calS^{d-1}(\bzero,1)$ and $k\in\intset{1}{d}$, then
using the definition \eqref{def_entropy}, we have
\begin{multline*}
\sum_{i=1}^m \partial_{v_i}\eta(\bv) \partial_{v_k}(\bef_i(\bv)\SCAL\bn - v_i \bsfW\SCAL \bn)
= \partial_{v_k}(\bq(\bv)\SCAL\bn)  
- \bsfW\SCAL \bn \sum_{i=1}^m \partial_{v_i}\eta(\bv) \partial_{v_k}(v_i) \\
= \partial_{v_k}(\bq(\bv)\SCAL\bn)
-  \bsfW\SCAL \bn \partial_{v_k}(\eta(\bv)) 
= \partial_{v_k}((\bq(\bv) - \eta(\bv)\bsfW)\SCAL\bn).
\end{multline*}
The conclusion follows readily.
\end{proof}

\section{Geometric preliminaries}
\label{Sec:Preliminaries}
In this section we introduce some notation and recall some general
results about Lagrangian mappings. The key results, which will be
invoked in \S\ref{Sec:ALE_scheme} and \S\ref{Sec:stability_analysis},
are lemmas~\ref{Lem:dtdetJ}, \ref{Lem:mass_transformation}, and
\ref{Lem:ALE_hyperbolic_system}.  The reader who is familiar with
these notions is invited to skip this section and to go directly to
\S\ref{Sec:ALE_scheme}.

\subsection{Jacobian of the coordinate transformation}
\label{Sec:Jacobian_coordinate_transformation}
Let $\bPhi:\Real^d \CROSS \Real_+ \longrightarrow \Real^d$ be a uniformly
Lipschitz mapping. We additionally assume that there is $t^*>0$ such
that the mapping $\bPhi_t: \Real^d\ni\bxi\longmapsto \bPhi_t(\bxi):=\bPhi(\bxi,t)\in\Real^d$ is
invertible for all $t\in[0,t^*]$. Let $\bvale:\Real^d\CROSS [0,t^*]\longrightarrow
\Real^d$ be the vector field implicitly defined by 
\begin{equation}
  \bvale(\bPhi(\bxi,t),t) := \partial_t \bPhi(\bxi,t), \quad \forall
  (\bxi,t)\in\Real\CROSS [0,t^*]. \label{def_of_v_abstract}
\end{equation}
Note that this definition makes sense owing to the inversibility
assumption on the mapping $\bPhi_t$; actually
\eqref{def_of_v_abstract} is equivalent to $\bvale(\bx,t) := \partial_t
\bPhi(\bPhi_t^{-1}(\bx),t)$ for any $t\in[0,t^*]$.

\begin{lemma}[Liouville's formula] \label{Lem:dtdetJ}
 Let
$\polJ(\bxi,t)=\GRAD_{\bxi}\bPhi(\bxi,t)$ be the Jacobian matrix of $\bPhi$,
then
\begin{equation}
\partial_t \det (\polJ(\bxi,t)) = (\DIV\bvale)(\bPhi(\bxi,t),t) \det
(\polJ(\bxi,t)).
\label{eq:Lem:dtdetJ}
\end{equation}
\end{lemma}
\begin{proof}
  This result is wellknown but we give the proof for completeness.
  Let $S_d$ be the set of all permutations of the set $\{1,\ldots,
  d\}$ and $\text{sgn}(\sigma)$ denotes the signature of $\sigma\in
  S_d$.  The Leibniz formula for the determinant gives
  $\det (\polJ(\bxi,t)) = \sum_{\sigma\in S_n}
  \text{sgn}(\sigma) \partial_{\sigma(1)}
  \Phi_1\ldots \partial_{\sigma(d)} \Phi_d$. Then using the definition
  of the field $\bvale$, with Cartesian coordinates $v_1,\ldots,v_d$,
  we have
\begin{multline*}
\partial_t \det(\polJ(\bxi,t)) 
= \sum_{\sigma\in S_n}\text{sgn}(\sigma) 
\partial_{\sigma(1)}\partial_t\Phi_1\ldots \partial_{\sigma(d)}\Phi_d
+\ldots+ \partial_{\sigma(1)}\Phi_1\ldots \partial_{\sigma(d)}\partial_t\Phi_d \\
= \sum_{\sigma\in S_n}\text{sgn}(\sigma) 
\partial_{\sigma(1)} (v_1(\bPhi(\bxi,t),t))\ldots \partial_{\sigma(d)}\Phi_d
+ \ldots+\partial_{\sigma(1)}\Phi_1\ldots \partial_{\sigma(d)}(v_d(\bPhi(\bxi,t),t))\\
 = \sum_{k=1}^d (\partial_k v_1)(\bPhi(\bxi,t),t) \sum_{\sigma\in S_n}\text{sgn}(\sigma) 
\partial_{\sigma(1)}\Phi_k \partial_{\sigma(2)}\Phi_2\ldots \partial_{\sigma(d)}\Phi_d
  +\\
 \ldots + \sum_{k=1}^d (\partial_k v_d)(\bPhi(\bxi,t),t) \sum_{\sigma\in S_n}\text{sgn}(\sigma) 
\partial_{\sigma(1)}\Phi_1 \ldots\partial_{\sigma(d-1)}\Phi_{d-1} \partial_{\sigma(d)}\Phi_k.
\end{multline*}
Upon observing that $\sum_{\sigma\in S_n}\text{sgn}(\sigma) 
\partial_{\sigma(1)}\Phi_k \partial_{\sigma(2)}\Phi_2\ldots \partial_{\sigma(d)}\Phi_d$
is zero unless $k=1$, \etc, we infer that
\[
\partial_t \det (\polJ(\bxi,t)) = (\partial_1 v_1)(\bPhi(\bxi,t),t)\det(\polJ(\bxi,t)) 
+\ldots (\partial_d v_d)(\bPhi(\bxi,t),t) \det(\polJ(\bxi,t)),
\]
which proves the result.
\end{proof}

\begin{remark}
Note that in  \eqref{eq:Lem:dtdetJ} 
the expression $(\DIV\bvale)(\bPhi(\bxi,t),t)$ should not be confused with 
$\DIV(\bvale(\bPhi(\bxi,t),t))$.
\end{remark}

\subsection{Mass transformation} \label{Sec:mass_transformation}
Let $\psi \in C_0^0(\Real^d;\Real)$ be a continuous compactly supported
function. We define $\varphi(\bx,t) := \psi(\bPhi_t^{-1}(\bx))$ for all
$t\in[0,t^*]$, \ie $\varphi(\bPhi(\bxi,t),t) = \psi(\bxi)$ for all
$(\bxi,t)\in\Real^d\CROSS[0,t^*]$. We want to compute $\int_{\Real^d}
\varphi(\bx,t)\diff \bx$ and relate it to $\int_{\Real^d}
\psi(\bxi)\diff \bxi$. 

\begin{lemma}\label{Lem:mass_transformation}
  Assume that $t\longmapsto \det(\polJ(\bxi,t))$ is a polynomial function in $t$ of
  degree at most $r\in\polN$ for any $\bxi\in \Real^d$.  Let
  $(\omega_l,\zeta_l)_{l\in\calL}$ be a quadrature such that $\int_0^1
  f(\zeta)\diff\zeta \simeq \sum_{l\in\calL} \omega_l f(\zeta_l)$ is
  exact for all polynomials of degree at most $\max(r-1,0)$, then,
  using the notation $\varphi(\bx,t) = \psi(\bPhi_t^{-1}(\bx))$, we
  have
\begin{equation}
\int_{\Real^d} \varphi(\bx,t)\diff \bx - \int_{\Real^d}
\psi(\bxi)\diff \bxi  = t \sum_{l\in\calL} \omega_l \int_{\Real^d} \varphi(\bx,t\zeta_l)
\DIV\bvale(\bx,t\zeta_l) \diff \bx \label{GCL_identity}
\end{equation}
\end{lemma}

\begin{proof}
Using the definitions  we infer that
\begin{align*}
\int_{\Real^d} \varphi(\bx,t)\diff \bx - \int_{\Real^d}
\psi(\bxi)\diff \bxi  &= \int_{\Real^d} \varphi(\bPhi_t(\bxi),t)
\det(\polJ(\bxi,t)) \diff \bxi -
\int_{\Real^d}\psi(\bxi)\diff \bxi\\
& = \int_{\Real^d} \psi(\bxi)
\left[\int_0^t \partial_\zeta \det(\polJ(\bxi,\zeta))  \diff \zeta\right]\diff \bxi.
\end{align*}
Since by assumption $\det(\polJ(\bxi,\zeta))$ is a polynomial of
degree at most $r$ in $\zeta$, and since the quadrature
$(\omega_l,\zeta_l)_{l\in\calL}$ is exact for all polynomials of
degree at most $\max(r-1,0)$, we infer that
\begin{align*}
\int_{\Real^d} \varphi(\bx,t)\diff \bx - \int_{\Real^d}
\psi(\bxi)\diff \bxi  &= t\sum_{l\in\calL} \omega_l \int_{\Real^d} \psi(\bxi)
\partial_\zeta \det(\polJ(\bxi,t\zeta_l)) \diff \bxi\\
&= t \sum_{l\in\calL} \omega_l \int_{\Real^d} \psi(\bxi)
(\DIV\bvale)(\bPhi(\bxi,t\zeta_l),t\zeta_l) \det(\polJ(\bxi,t\zeta_l)) \diff \bxi \\
&= t \sum_{l\in\calL} \omega_l \int_{\Real^d} \varphi(\bx,t)
\DIV\bvale(\bx,t\zeta_l) \diff \bx,
\end{align*}
where we used Lemma~\ref{Lem:dtdetJ}.
\end{proof}

\begin{remark}
  The statement of Lemma~\ref{Lem:mass_transformation} is somewhat
  similar in spirit to Eq.~(26) in
  \cite{Farhat_Geuzaine_Grandmont_2001} or Eq.~(7) in
  \cite{Guillard_Farhat_2000}. The identity~\eqref{GCL_identity} will
  allow us to prove a statement
  that is known in the ALE literature as the Discrete Geometric
  Conservation Law (DGCL).
\end{remark}

\subsection{Arbitrary Lagrangian Eulerian formulation}
\label{Sec:ALE} Let $\bPhi:\Real^d \CROSS
\Real_+ \longrightarrow \Real^d$ be a uniformly Lipschitz mapping as
defined in \S\ref{Sec:Jacobian_coordinate_transformation} and let
$[0,t^*]$ be the interval where the mapping $\Real^d\ni\bxi\longmapsto
\bPhi(\bxi,t)\in\Real^d$ is invertible for all $t\in[0,t^*]$. Let
$\bvale$ be the vector field defined in \eqref{def_of_v_abstract}, \ie
$\bvale(\bx,t) = \partial_t \bPhi(\bPhi_t^{-1}(\bx),t)$, and let $\bu$ be
a weak solution to \eqref{def:hyperbolic_system}. The following result
is the main motivation for the arbitrary Lagrangian Eulerian
formulation that we are going to use in the paper.
\begin{lemma} \label{Lem:ALE_hyperbolic_system} The following identity
  holds in the distribution sense (in time) over the interval $[0,t^*]$ for
  every function $\psi\in C^0_0(\Real^d;\Real)$ (with the notation
  $\varphi(\bx,t):=\psi(\bPhi_t^{-1}(\bx))$):
\begin{equation}
\partial_t \int_{\Real^d} \bu(\bx,t) \varphi(\bx,t) \diff \bx 
= \int_{\Real^d} \DIV(\bu\otimes\bvale -\bef(\bu)) \varphi(\bx,t)\diff \bx.
\label{pde_ale_weak_formulation}
\end{equation}
\end{lemma}

\begin{proof}
Using the chain rule and Lemma~\ref{Lem:dtdetJ}, we have
\begin{multline*}
\partial_t \int_{\Real^d} \bu(\bx,t) \varphi(\bx,t) \diff \bx 
= \partial_t \int_{\Real^d} \bu(\bPhi_t(\bxi),t) \det(\polJ(\bxi,t)) \psi(\bxi) \diff \bxi \\
= \int_{\Real^d} \big\{\partial_t (\bu(\bPhi_t(\bxi),t))\det(\polJ(\bxi,t))
+  \bu(\bPhi_t(\bxi),t)\partial_t(\det(\polJ(\bxi,t))) \big\}\psi(\bxi) \diff \bxi \\
= \int_{\Real^d} \big\{(\partial_t\bu)(\bPhi_t(\bxi),t)
+ \partial_t\bPhi(\bxi,t)\SCAL(\GRAD \bu)(\bPhi_t(\bxi),t) \big\}\det(\polJ(\bxi,t))\psi(\bxi)\diff \bxi  \\
+ \int_{\Real^d} \bu(\bPhi_t(\bxi),t) (\DIV\bvale)(\bPhi_t(\bxi),t) \det
(\polJ(\bxi,t))\psi(\bxi) \diff \bxi.
\end{multline*}
Then using \eqref{def:hyperbolic_system} and the definition
of the vector field $\bvale$ yields
\begin{multline*}
\partial_t \int_{\Real^d} \bu(\bx,t) \varphi(\bx,t) \diff \bx 
= \int_{\Real^d}
-\DIV\bef(\bu)(\bPhi_t(\bxi),t)\psi(\bxi)\det(\polJ(\bxi,t))\diff \bxi
\\ 
+  \int_{\Real^d}\!\!\big\{\bvale(\bPhi_t(\bxi),t)\SCAL(\GRAD\bu)(\bPhi_t(\bxi),t)
+ (\DIV\bvale)(\bPhi_t(\bxi),t)\bu(\bPhi_t(\bxi),t)\big\}\psi(\bxi)\det(\polJ(\bxi,t))\diff \bxi  \\
= \int_{\Real^d} \big\{\!-\DIV\bef(\bu)(\bPhi_t(\bxi),t) +\DIV(\bu\otimes\bvale)(\bPhi_t(\bxi),t)\big\}
\psi(\bxi)\det(\polJ(\bxi,t))\diff \bxi.
\end{multline*}
We conclude by making the change of variable $\bx= \bPhi(\bxi,t)$.
\end{proof}

We now state a result regarding the notion of entropy solution in the
ALE framework. The proof of this result is similar to that of
Lemma~\ref{Lem:ALE_hyperbolic_system} and is therefore omitted for brevity.
\begin{lemma} \label{Lem:ALE_hyperbolic_system_entropy} Let
  $(\eta,\bq)$ be an entropy pair for \eqref{def:hyperbolic_system}.
  The following inequality holds in the distribution sense (in time)
  over the interval $[0,t^*]$ for every non-negative function $\psi\in
  C^0_0(\Real^d;\Real_+)$ (with the notation
  $\varphi(\bx,t):=\psi(\bPhi_t^{-1}(\bx))$):
\begin{equation}
\partial_t \int_{\Real^d} \eta(\bu(\bx,t)) \varphi(\bx,t) \diff \bx 
\le \int_{\Real^d} \DIV(\eta(\bu)\bvale -\bq(\bu)) \varphi(\bx,t)\diff \bx.
\label{pde_ale_weak_formulation_entropy}
\end{equation}
\end{lemma}

\section{The Arbitrary Lagrangian Eulerian
  algorithm} \label{Sec:ALE_scheme} We describe in this section the
ALE algorithm to approximate the solution of
\eqref{def:hyperbolic_system}. We use continuous finite elements and
explicit time stepping. We are going to use two different discrete
settings: one for the mesh motion and one for the approximation of
\eqref{def:hyperbolic_system}.

\subsection{Geometric finite elements and mesh} \label{Sec:Geometric_FE_and_Mesh}
Let $(\calT_h^0)_{h>0}$ be a shape-regular sequence of matching
meshes.  The symbol ${}^0$ in $\calT_h^0$ refers to the initial
configuration of the meshes.  The meshes will deform over time, in a way
that has yet to be defined, and we are going to use the symbol ${}^n$
to say that $\calT_h^n$ is the mesh at time $t^n$ for a given $h>0$.  We assume that the
elements in the mesh cells are generated from a finite number of
reference elements denoted $\wK_1,\dots,\wK_\varpi$.  For instance,
$\calT_h^0$ could be composed of a combination of triangles and
parallelograms in two space dimensions ($\varpi=2$ in this case); the
mesh $\calT_h^0$ could also be composed of a combination of
tetrahedra, parallelepipeds, and triangular prisms in three space
dimensions ($\varpi=3$ in this case). The diffeomorphism mapping
$\wK_r$ to an arbitrary element $K\in \calT_h^n$ is denoted $T_K^n :
\wK_r \longrightarrow K$ and its Jacobian matrix is denoted
$\polJ_K^n$, $1\le r\le \varpi$. We now introduce a set of reference
Lagrange finite elements
$\{(\wK_r,\wP_r\upgeo,\wSigma_r\upgeo)\}_{1\le r\le \varpi}$ (the
index $r\in\intset {1}{\varpi}$ will be omitted in the rest of the
paper to alleviate the notation).  Letting $\nfgeo:=\dim\wP\upgeo$, we
denote by $\{\wba_i\}_{i\in\intset{1}{\nfgeo}}$ and
$\{\wtheta_i\upgeo\}_{i\in\intset{1}{\nfgeo}}$ the Lagrange nodes of
$\wK$ and the associated Lagrange shape functions.

 The sole purpose of the
geometric reference element $\{(\wK,\wP\upgeo,\wSigma\upgeo)$ is to
construct the geometric transformation $T_K^n$ as we now explain. Let
$\{\ba_i^n\}_{i\in\intset{1}{\Nglobgeo}}$ be the collection of all the
Lagrange nodes in the mesh $\calT_h^n$, which we organize in cells by
means of the geometric connectivity array $\jjgeo :
\calT_h^n\CROSS\intset{1}{\nfgeo} \longrightarrow
\intset{1}{\Nglobgeo}$ (assumed to be independent of the time index
$n$). Given a mesh cell $K\in \calT_h^n$, the connectivity array is
defined such that $\{\ba_{\jjgeo(i,K)}^n\}_{i\in\intset{1}{\nfgeo}}$
is the set of the Lagrange nodes describing $K^n$. More precisely, 
upon defining the geometric transformation $T_K^n : \wK\longrightarrow K$  at time $t^n$ by
\begin{equation}
T_K^n(\wbx) = \sum_{i\in\intset{1}{\nfgeo}} \ba_{\jjgeo(i,K)}^n
\wtheta_i\upgeo(\wbx)
\label{def_geometric_transformation}
\end{equation}
we have $K:=T_K^n(\wK)$.  In other words the geometric transformation
is {\em fully described by the motion of geometric Lagrange nodes.} We
finally recall that constructing the Jacobian matrix $\polJ_K^n$ from
\eqref{def_geometric_transformation} is an elementary operation for
any finite element code.

\subsection{Approximating finite elements}
We now introduce a set of reference finite elements
$\{(\wK_r,\wP_r,\wSigma_r)\}_{1\le r\le \varpi}$ which we are going to use to construct an
approximate solution to \eqref{def:hyperbolic_system} (the index
$r\in\intset {1}{\varpi}$ will be omitted in the rest of the paper to
alleviate the notation).
The shape functions on the reference element 
are denoted
$\{\wtheta_i\}_{i\in\intset{1}{\nf}}$.
We assume that the basis $\{\wtheta_i\}_{i\in\intset{1}{\nf}}$ has the
following key properties:
\begin{align}
\wtheta_i(\bx) \ge 0, \quad \sum_{i\in\intset{1}{\nf}}\wtheta_i(\wbx) =1,\
\quad \forall \wbx \in\wK.  \label{reference_postivity_unity}
\end{align}
These properties hold for linear Lagrange elements.  It holds true
also for Bernstein-Bezier finite elements, see \eg
\cite[Chap.~2]{Lai_Schumaker_2007}, \cite{Ainsworth_2014}. 

Given the mesh $\calT_h^n$, we denote by $\Dom^n$ the computational
domain generated by $\calT_h^n$ and define the scalar-valued space
\begin{align} \label{eq:Xh}
P(\calT_h^n) &:=\{ v\in \calC^0(\Dom^n;\Real)\st 
v_{|K}{\circ}T_K^n \in \wP,\ \forall K\in \calT_h^n\},
\end{align} 
where $\wP$ is the reference polynomial space. 
We also introduce the vector-valued spaces
\begin{align} \label{eq:Xh_vector}
\bP_d(\calT_h^n) :=[P(\calT_h^n)]^d, 
\quad \text{and} \quad
\bP_m(\calT_h^n) :=[P(\calT_h^n)]^m.
\end{align}
We are going to approximate the ALE velocity in $\bP_d(\calT_h^n)$ and
the solution of \eqref{def:hyperbolic_system} in $\bP_m(\calT_h^n)$.
The global shape functions in $P(\calT_h^n)$ are denoted by
$\{\psi_i^n\}_{i\in\intset{1}{\Nglob}}$.  Recall that these functions
form a basis of $P(\calT_h^n)$. Let $\jj :
\calT_h^n\CROSS\intset{1}{\nf} \longrightarrow \intset{1}{\Nglob}$ be
the connectivity array, which we assume to be independent of $n$. This
array is defined such that
\begin{equation}
\psi_{\jj(i,K)}^n(\bx) = \wtheta_i((T_K^n)^{-1}(\bx)), \quad \forall
i\in\intset{1}{\nf},\ \forall K\in \calT_h^n.
\label{def_local_shape_functions}
\end{equation}
This definition together with \eqref{reference_postivity_unity}
implies that
\begin{align}
\psi_i^n(\bx) \ge 0, \quad \sum_{i\in\intset{1}{\Nglob}}\psi_i^n(\bx) =1,\
\quad \forall \bx \in\Real^d.  \label{postivity_unity}
\end{align}

We denote by $S_i^n$ the support of $\psi_i^n$ and by $\mes{S_i^n}$
the measure of $S_i$, $i\in\intset{1}{\Nglob}$. We also define
$S_{ij}^n:=S_i^n\cap S_j^n$ the intersection of the two supports
$S_i^n$ and $S_j^n$. Let $E$ be a union of cells in $\calT_h^n$; we
define $\calI(E):=\{j\in\intset{1}{\Nglob} \st \mes{S_j^n \cap
  E}\not=0\}$ the set that contains the indices of all the shape
functions whose support on $E$ is of nonzero measure.  Note that the
index set $\calI(E)$ does not depend on the time index $n$ since we
have assumed that the connectivity of the degrees of freedom is fixed
once for all.  We are going to regularly invoke $\calI(K)$ and
$\calI(S_i^n)$ and the partition of unity property:
$\sum_{i\in\calI(K)} \psi_i^n(\bx) =1$ for all $\bx\in K$.

\begin{lemma} \label{Lem:convexity} For all $K\in \calT_h^n$, all $\bx\in K$, and all
  $\bv_h:=\sum_{i\in\intset{1}{\Nglob}} \bsfV_i \psi_i^n \in \bP_m(\calT_h^n)$, $\bv_h(\bx)$ is
  in the convex hull of $(\bsfV_i)_{i\in \calI(K)}$ (henceforth
  denoted $\mathrm{conv}(\bsfV_i)_{i\in \calI(K)}$). Moreover for any
  convex set $A$ in $\Real^m$, we have
\begin{equation}
  \left((\bsfV_i)_{i\in \calI(K)} \in A \right) 
\Rightarrow 
\left(\bv_h(\bx) \in A, \ \forall \bx \in K\right).
  \label{mesh_convexity_assumption} 
\end{equation}
\end{lemma}
\begin{proof} The positivity and partition of unity
  assumption~\eqref{postivity_unity} and the definition $\bv_h(\bx) =
  \sum_{i\in\calI(K)} \bsfV_i \psi_i^n(\bx)$ implies that
  $\bv_h(\bx)$ is a convex combination of
  $(\bsfV_i)_{i\in\calI(K)}$, whence the conclusion.  The
  statement \eqref{mesh_convexity_assumption} follows readily since
  the convexity assumption on $A$ implies that
  $\mathrm{conv}(\bsfV_i)_{i\in \calI(K)} \subset A$.
\end{proof}

Let $\calM^n\in \Real^{\Nglob\CROSS \Nglob}$ be the
consistent mass matrix with entries
$\int_{S_{ij}^n} \psi_i^n(\bx)\psi_j^n(\bx)\diff \bx$, and let $\calM^{L,n}$
be the diagonal lumped mass matrix with entries
\begin{equation}
  m_i^n:= \int_{S_i^n} \psi_i^n(\bx)\diff \bx. \label{def_of_mi}
\end{equation}
The partition of unity property implies that $m^n_i = \sum_{j\in
  \calI(S_i^n)} \int \psi_j^n(\bx)\psi_i^n(\bx)\diff \bx$, \ie the
entries of $\calM^{L,n}$ are obtained by summing the rows of $\calM^n$.
Note that the positivity assumption~\eqref{postivity_unity} implies
that $m_i^n>0$ for any $i\in\intset{1}{\Nglob}$.

\subsection{The ALE algorithm, version 1} \label{Sec:alg1} Let
$\calT_h^0$ be the mesh at the initial time $t=0$.  Let
$(\fm_i^{0})_{i\in\intset{1}{\Nglob}}$ be the approximations of
  the mass of the shape functions at time $t^0$ defined by $\fm_i^0 =
m_i^0:=\int_{\Real^d} \psi_i^0(\bx) \diff \bx$.  Let $\bu_{h0}:=
\sum_{i\in\intset{1}{\Nglob}}\bsfU_{i}^0\psi_i^0 \in \bP_m(\calT_h^0)$
be a reasonable approximation of the initial data $\bu_0$ (we shall
make a more precise statement later).

Let $\calT_h^n$ be the mesh at time $t^n$, $(\fm_i^n)_{1\le i\le
  \Nglob}$ be the approximations of the mass of the shape functions at
time $t^n$, and $\bu_h^n:= \sum_{i\in\intset{1}{\Nglob}}
\bsfU_i^n\psi_i^n \in\bP_m(\calT_h^n)$ be the approximation of $\bu$
at time $t^n$.  We denote by $\fM^{L,n}$ the approximate lumped
matrix, \ie $\fM^{L,n}_{ij} = \fm_i^n \delta_{ij}$. We now make the
  assumption that the given ALE velocity field is a member of
  $\bP_d(\calT_h^n)$, \ie $\bw^n = \sum_{i\in\intset{1}{\Nglob}}
\bsfW_i^n \psi_i^n\in\bP_d(\calT_h^n)$. Then the Lagrange nodes of the
mesh are moved by using
\begin{equation}
\ba_i^{n+1} = \ba_i^n + \dt \bw^n(\ba_i^n). \label{alg1_motion_of_ai}
\end{equation}
This fully defines the mesh $\calT_h^{n+1}$ as explained at the end of
\S\ref{Sec:Geometric_FE_and_Mesh}.  We now estimate the mass of the
shape function $\psi_i^{n+1}$.  Of course we could use
$m_{i}^{n+1}=\int_{\Real^d} \psi_i^{n+1}(\bx) \diff \bx$, this option
will be explored in \S\ref{sec:alternative}, but to make the method
easier to extend with higher-order strong stability preserving (SSP)
time stepping techniques, we define $\fm_{i}^{n+1}$ by approximating
\eqref{GCL_identity} with a first-order quadrature rule,
\begin{equation}
\fm_i^{n+1} = \fm_i^n +\dt \int_{S_i^n} \psi_i^n(\bx) \DIV \bw^n(\bx) \diff \bx.
\label{alg1_def_of_fmi}
\end{equation}
Taking inspiration from
\eqref{pde_ale_weak_formulation}, we propose to compute $\bu_h^{n+1}$
by using the following explicit technique:
 \begin{multline}
\frac{\fm_i^{n+1} \bsfU_i^{n+1}-\fm_i^{n} \bsfU_i^n}{\dt}  
- \sum_{j\in \calI(S_i^n)} d_{ij}^{n} \bsfU^n_j \\
+
\int_{\Real^d} \DIV\bigg(\sum_{j\in\intset{1}{\Nglob}}
(\bef(\bsfU_j^n)  - \bsfU_j^n\otimes \bsfW_j^{n})\psi_j^n(\bx)\bigg) 
\psi_i^n(\bx) \diff \bx  
 = 0,  \label{def_of_scheme_dij}
\end{multline}
where $\bu_h^{n+1}:= \sum_{i\in\intset{1}{\Nglob}}
\bsfU_i^{n+1}\psi_i^{n+1} \in\bP_m(\calT_h^{n+1})$. Notice that we have
replaced the consistent mass matrix by an approximation of the lumped
mass matrix to approximate the time derivative.  The coefficient
$d_{ij}^{n}$ is an artificial viscosity for the pair of degrees of
freedom $(i,j)$ that will be identified by proceeding as in
\cite{Guermond_Popov_Hyp_2015}. We henceforth assume that
$d_{ij}^{n}=0$ if $j\not\in\calI(S_i^n)$ and
\begin{equation}
  d_{ij}^{n}\ge 0, \ \text{if}\   i\not= j,\quad
  d_{ij}^{n}=d_{ji}^{n},\quad  \text{and} \quad  d_{ii}:=\sum_{i\ne j\in\calI(S_i^n)}  -d_{ji}^{n}.
\label{introduction_dij}
\end{equation} 
The entire process is described in Algorithm~\ref{alg1}.

Let us reformulate \eqref{def_of_scheme_dij} in a form that is more
suitable for computations. Let us introduce the vector-valued coefficients
\begin{equation}
\bc_{ij}^n := \int_{S_i^n} \GRAD\psi_j^n(\bx)
\psi_i^n(\bx)\diff \bx.
\label{def_of_cij}
\end{equation}
We define the unit vector $\bn_{ij}^{n} := \frac{\bc_{ij}^{n}}{\|\bc_{ij}^{n}\|_{\ell^2}} $.
Then we can rewrite \eqref{def_of_scheme_dij} as follows
\begin{align}
\frac{\fm_i^{n+1} \bsfU_i^{n+1}-\fm_i^{n} \bsfU_i^n}{\dt} 
+ \sum_{j\in \calI(S_i^n)} (\bef(\bsfU_j^n) -\bsfU_j^n\otimes\bsfW_j^{n})\SCAL \bc_{ij}^{n}
 -  d_{ij}^{n} \bsfU^n_j  = 0.  \label{def_of_scheme_cij_dij} 
\end{align}
It will be shown in the proof of Theorem~\ref{Thm:invariant_domain}
that an admissible choice for $d_{ij}^{n}$ is
\begin{equation}
d_{ij}^{n} = 
\max(\lambda_{\max}(\bg_{j}^n,\bn_{ij}^n,\bsfU_i^n,\bsfU_j^n) \|\bc_{ij}^{n}\|_{\ell^2},
\lambda_{\max}(\bg_{i}^n,\bn_{ji}^n,\bsfU_j^n,\bsfU_i^n) \|\bc_{ji}^{n}\|_{\ell^2}).
 \label{def_of_dij}
\end{equation}
where $\lambda_{\max}(\bg_j^n,\bn_{ij}^n,\bsfU_i^n,\bsfU_j^n)$ 
is the largest wave speed in the following one-dimensional Riemann
problem:
\begin{equation}
\partial_t \bv + \partial_x (\bg_j^n(\bv)\SCAL\bn_{ij}^n) =0,\quad (x,t)\in
\Real\CROSS\Real_+, 
\quad \bv(x,0)=
\begin{cases} 
\bsfU_i^{n} & \text{if $x<0$}\\
\bsfU_j^{n} & \text{if $x>0$}.
\end{cases} \label{one_D_Riemann_problem}
\end{equation}
where we have defined the flux $\bg_j^n(\bv) := \bef(\bv) - \bv\otimes\bsfW_j^{n}$.
\begin{remark}[Fastest wave speed]
  The fastest wave speed in \eqref{one_D_Riemann_problem} can be
  obtained by estimating the fastest wave speed in the Riemann problem
  \eqref{def:Riemann_problem} with the flux $\bef(\bv)\SCAL\bn_{ij}^n$
  and initial data $(\bsfU_i^n,\bsfU_j^n)$.  Let
  $\lambda_L(\bef,\bn_{ij}^n,\bsfU_i^n,\bsfU_j^n)$ and
  $\lambda_R(\bef,\bn_{ij}^n,\bsfU_i^n,\bsfU_j^n)$ be the speed of the
  leftmost and rightmost waves in \eqref{def:Riemann_problem},
  respectively.  Then 
\begin{multline}
 \lambda_{\max}(\bg_{j}^n,\bn_{ij}^n,\bsfU_i^n,\bsfU_j^n)
  = \max(|\lambda_L(\bef,\bn_{ij}^n,\bsfU_i^n,\bsfU_j^n) -
  \bsfW_j^{n}\SCAL\bn_{ij}^n|,\\
  |\lambda_R(\bef,\bn_{ij}^n,\bsfU_i^n,\bsfU_j^n) -
  \bsfW_j^{n}\SCAL\bn_{ij}^n|). \label{lambda_max_shifted}
\end{multline} 
A very fast algorithm to compute $\lambda_L(\bef,\bn_{ij}^n,\bsfU_i^n,\bsfU_j^n)$ and
  $\lambda_R(\bef,\bn_{ij}^n,\bsfU_i^n,\bsfU_j^n)$ for the compressible Euler equations
is described in \cite{Guermond_Popov_Fast_Riemann_2016}; see also \cite{Toro_2009}.
\end{remark}

\begin{algorithm}
  \caption{} \label{alg1}
  \begin{algorithmic}[1]
    \REQUIRE $\bu_h^{0}$ and $\fM^{L,0}$ \WHILE{$t^n<T$} \STATE Use
    CFL condition to estimate $\tau$.  \IF{$t^n+\tau> T$} \STATE
    $\tau\gets T-t^n$
    \ENDIF
    \STATE Estimate/choose $\bw^{n}$ and make sure that the
    transformation $\bPhi_t$ defined in \eqref{def_of_Phit} is
    invertible over the interval $[t^n,t^{n+1}]$.  
    \STATE Move mesh from $t^n$ to $t^{n+1}$ using
    \eqref{alg1_motion_of_ai}.
    \STATE Compute $\fm_i^{n+1}$, see \eqref{alg1_def_of_fmi}. Check
     $\fm_i^{n+1}>0$; otherwise,  go to step 6, reduce $\dt$.
    \STATE Compute $\bc_{ij}^n$ as in \eqref{def_of_cij}.
    \STATE Compute $d_{ij}^{n}$, see \eqref{def_of_dij} and
    \eqref{introduction_dij}.  
    \STATE Check $1-\sum_{i\ne j\in
      \calI(S_i^n)} 2d_{ij}^{n}\frac{\dt}{\fm_i^{n+1}}$ positive.
    Otherwise, go to step 6 and reduce $\dt$.  
    \STATE Compute $\bu_h^{n+1}$ by using
    \eqref{def_of_scheme_cij_dij}.
    \STATE $t^n \gets t^n+\tau$
    \ENDWHILE
  \end{algorithmic}
\end{algorithm}

Since it will be important to compare $\bsfU_j^{n+1}$ and $\bsfU_j^n$
to establish the invariant domain property, we rewrite the scheme in a
form that is more suitable for this purpose.
\begin{lemma}[Non-conservative form] \label{Lem:one_mi}
The scheme \eqref{def_of_scheme_dij} is equivalent to
\begin{align}
  \fm_i^{n+1} \frac{\bsfU_i^{n+1}-\bsfU_i^n}{\dt} &= \sum_{j\in
    \calI(S_i^n)}
  ((\bsfU_j^n-\bsfU_i^n)\otimes\bsfW_j^{n}-\bef(\bsfU_j^n))\SCAL
  \bc_{ij}^{n} + d_{ij}^{n} \bsfU^n_j,\label{Eq:Lem:one_mi}
\end{align}
\end{lemma}
\begin{proof}
We rewrite
\eqref{def_of_scheme_cij_dij}
as follows:
\begin{align*}
  \fm_i^{n+1} \frac{\bsfU_i^{n+1}-\bsfU_i^n}{\dt} + \frac{\fm_i^{n+1}
    -\fm_i^{n}}{\dt} \bsfU_i^n = \sum_{j\in \calI(S_i^n)}
  (\bsfU_j^n\otimes\bsfW_j^{n}-\bef(\bsfU_j^n))\SCAL \bc_{ij}^{n} +
  d_{ij}^{n} \bsfU^n_j,
\end{align*}
Then, recalling the expression $\bw^n= \sum_{i\in\intset{1}{\Nglob}}
\bsfW_i^{n} \psi_i^n$, and using \eqref{alg1_def_of_fmi}, we infer
that $\fm_i^{n+1} = \fm_i^{n} + \dt \sum_{j\in \calI(S_i^n)} \bsfW_j^n \SCAL \bc_{ij}^n$,
which in turn implies that
\begin{align*}
 (\fm_i^{n+1} - \fm_i^{n})\bsfU_i^n
& = \dt \bsfU_i^n \sum_{j\in\calI(S_i^n)} \bsfW_j^{n}\SCAL\bc_{ij}^{n}
= \dt  \sum_{j\in\calI(S_i^n)} (\bsfU_i^n\otimes\bsfW_j^{n})\SCAL\bc_{ij}^{n},
\end{align*}
whence the result.
\end{proof}

\begin{remark}[Other discretizations]
  Note that the method for computing the artificial diffusion is quite
  generic, \ie it is not specific to continuous finite elements.  The
  above method can be applied to any type of discretization that can
  be put into the form~\eqref{def_of_scheme_cij_dij}.
\end{remark}

\subsection{Continuous mesh motion} We introduce in this section some
technicalities regarding the mesh motion that will be used in the
second version of the algorithm and which will be described in
\S\ref{sec:alternative}. Our main motivation is to replace the
approximate mass conservation \eqref{alg1_def_of_fmi} by the exact
quadrature~\eqref{GCL_identity}. For this purpose, we need to consider
the continuous motion of the mesh over the time interval
$[t^n,t^{n+1}]$.

Given a mesh $\calT_h^n$ we denote by $\Dom^n$ the computational
domain generated by $\calT_h^n$. Then using the standard constructions of
continuous finite element spaces, we define a new scalar-valued space
based on the geometric Lagrange finite elements
$\{(\wK_r,\wP_r\upgeo,\wSigma_r\upgeo)\}_{1\le r\le \varpi}$:
\begin{align} \label{eq:Xhgeo}
P\upgeo(\calT_h^n) &=\{ v\in \calC^0(\Dom^n;\Real)\st 
v_{|K}{\circ}T_K^n \in \wP\upgeo,\ \forall K\in \calT_h^n\},
\end{align} 
where $\wP\upgeo$ is the reference polynomial space defined on $\wK$ (note
that the index $r$ has been omitted). 
We also introduce the vector-valued spaces
\begin{align} \label{eq:Xhgeo_vector}
\bP\upgeo(\calT_h^n) :=[P\upgeo(\calT_h^n)]^d.
\end{align} 
We denote by
$\{\psi_i\upgeon\}_{i\in\intset{1}{\Nglob}}$ the global shape functions
in $P\upgeo(\calT_h^n)$.  Recall that $\{\psi_i\upgeon\}_{i\in\intset{1}{\Nglob}}$
is a basis of $P\upgeo(\calT_h^n)$ and
\begin{equation}
\psi_{\jjgeo(i,K)}\upgeon(\bx) = \wtheta_i\upgeo((T_K^n)^{-1}(\bx)), \quad \forall
i\in\intset{1}{\nf},\ \forall K\in \calT_h^n, \ \forall \bx\in K.
\label{def_local_geo_shape_functions}
\end{equation}

The key difference with version~1 of the algorithm is that now
  we are going to construct exactly the mapping $\bPhi_t:\calT_h^n
  \longrightarrow \calT_h(t)$ by using $\bP\upgeo(\calT_h^n)$, and we
  are going to assume that the given ALE velocity is a member of
  $\bP\upgeo(\calT_h^n)$ instead of $\bP_d(\calT_h^n)$ as we did in
  \S\ref{Sec:alg1}.
Let $\bw^{n} = \sum_{i\in\calI(K)} \bsfW_i\upgeon \psi_i\upgeon\in
\bP\upgeo(\calT_h^n)$ be the ALE velocity.
 
Let us construct the associated transformation $\bPhi_t:
\Real^d\longrightarrow \Real^d$ and velocity field $\bvale$ for any
$t\in [t^n,t^{n+1}]$.  We define a continuous deformation of the mesh
over the time interval $[t^n,t^{n+1}]$ by moving the nodes
$\{\ba^n_i\}_{i\in\intset{1}{\Nglobgeo}}$ as follows:
\begin{equation}
\ba_i(t) = \ba_i^n + t \bsfW_i\upgeon ,\qquad t\in [t^n,t^{n+1}]. \label{motion_of_ai}
\end{equation}
This rule completely defines the mesh $\calT_h(t)$ owing to the
definition of the geometric transformation $T_K(t) :\wK
\longrightarrow K$, with $T_K(t)(\wbx) = \sum_{i\in\intset{1}{\nfgeo}}
\ba_{\jjgeo(i,K)}(t) \wtheta_i\upgeo(\wbx)$, for all $K\in
\calT_h(t)$, see \eqref{def_geometric_transformation}.  The shape
functions of $P\upgeo(\calT_h(t))$ and $P(\calT_h(t))$ are defined as
usual by setting 
\begin{align}
\varphi_{\jj(i,K)}\upgeo(\bx,t)&:=
\wtheta_i\upgeo((T_K(t))^{-1}(\bx)),\quad \forall\bx\in K,\ \forall K \in\calT_h(t),
\ \forall i\in\intset{1}{\Nglobgeo}\label{def_local_shape_functions_geo_t}\\
 \varphi_{\jj(i,K)}(\bx,t)&:=
\wtheta_i((T_K(t))^{-1}(\bx)),\quad \forall\bx\in K,\ \forall K\in\calT_h(t),
\ \forall i\in\intset{1}{\Nglob}.\label{def_local_shape_functions_t}
\end{align}  
We recall that
$\calT_h^{n+1}:=\calT_h(t^{n+1})$. Notice that
$\varphi_i\upgeo(\cdot,t^n) = \psi_i\upgeon$ and
$\varphi_i\upgeo(\cdot,t^{n+1}) = \psi_i\upgeonpone$ for any
$i\in\intset{1}{\Nglobgeo}$, and $\varphi_i(\cdot,t^n) = \psi_i^n$ and
$\varphi_i(\cdot,t^{n+1}) = \psi_i^{n+1}$ for any
$i\in\intset{1}{\Nglob}$.

For any $t\in [t^n,t^{n+1}]$ 
we now define the mapping 
$\bPhi_t : \calT_h^n \longrightarrow \calT_h(t)$ by
\begin{equation}
  \bPhi_t(\bxi)_{|K} = \sum_{i\in\intset{1}{\nfgeo}} 
  \ba_{\jjgeo(i,K)}(t)\psi_{\jjgeo(i,K)}\upgeon(\bxi),
  \qquad \forall \bxi\in K,\ \forall \in \calT_h^n. \label{def_of_Phit}
\end{equation}

 \begin{lemma}\label{Lem:TKt_varphit}
   The following properties hold for any $K\in \calT_h^n$,
   any $t\in [t^n,t^{n+1}]$ and any $i\in\intset{1}{\Nglob}$:
\begin{align}
&T_K(t) = \bPhi_t\circ T_K^n,\\
&\varphi_i(\bPhi_t(\bxi),t) = \psi_i^n(\bxi), \quad \varphi_i\upgeo(\bPhi_t(\bxi),t) = \psi_i\upgeon(\bxi),
\quad \forall \bxi\in K\in \calT_h^n. \label{Eq:Lem:varphit}\\
&\bvale(\bx,t) = \sum_{i\in\intset{1}{\Nglob}}  \bsfW_i\upgeon  \varphi_i\upgeo(\bx,t),
\qquad  \forall t\in [t^n,t^{n+1}],\ \forall\bx \in \Real^d. \label{Eq:Lem:TKt_varphit_bvale}
\end{align}
\end{lemma}
\begin{proof}
  Let us observe first that the definition \eqref{def_of_Phit}
  together with the definitions of $T_K(t)$ and
  $\psi_{\jjgeo(i,K)}\upgeon = \wtheta_i\upgeo\circ (T_K^n)^{-1}$
  implies that
\begin{align*}
 \bPhi_t(\bxi)_{|K} &=  
\sum_{i\in\intset{1}{\nfgeo}}\!\! \ba_{\jjgeo(i,K)}(t)\psi_{\jjgeo(i,K)}\upgeon(\bxi)
=  \sum_{i\in\intset{1}{\nfgeo}}\!\!  \ba_{\jjgeo(i,K)}(t)\wtheta_i\upgeo\circ (T_K^n)^{-1}(\bxi),
\end{align*}
which implies that $\bPhi_{t|K} = T_K(t)\circ (T_K^n)^{-1}$.  This
proves the first statement. Second, the definition of the shape
functions \eqref{def_local_shape_functions_t} together with the above
result implies that
\begin{align*}
\varphi_{\jj(i,K)}(\bx,t)&= \wtheta_i((T_K(t))^{-1}(\bx)) =
\wtheta_i((\bPhi_t\circ T_K^n)^{-1}(\bx)) = 
\wtheta_i((T_K^n)^{-1}\circ(\bPhi_t)^{-1}(\bx)).
\end{align*}
This proves that $\varphi_{\jj(i,K)}(\bPhi_t(\bxi),t) =
\psi_{\jj(i,K)}^n(\bxi)$ for every $\bxi\in K\in\calT_h^n$.  
Proceed similarly to prove $\varphi_{\jj(i,K)}\upgeo(\bPhi_t(\bxi),t) =
\psi_{\jj(i,K)}\upgeon(\bxi)$.
Now let
us compute $\partial_t \bPhi$. Using the definition of the motion of
the nodes \eqref{motion_of_ai} and the definition of $\bPhi$,
\eqref{def_of_Phit}, we infer that
\[
\partial_t \bPhi(\bxi,t)_{|K} = \sum_{i\in\intset{1}{\nfgeo}}
\bsfW_{\jjgeo(i,K)}\upgeon \psi_{\jjgeo(i,K)}\upgeon(\bxi)= \bw^n_{|K}.
\]
Hence the definition of $\bvale$ gives
\[
\bvale(\bx,t) = \partial_t \bPhi(\bPhi_t^{-1}(\bx),t) = \bw^{n}(\bPhi_t^{-1}(\bx))
=\sum_{i\in\intset{1}{\Nglobgeo}}  \bsfW_i\upgeon  \psi_i\upgeon(\bPhi_t^{-1}(\bx)).
\]
We then conclude by invoking ~\eqref{Eq:Lem:varphit}, \ie 
$\varphi_i\upgeo(\bx,t) = \psi_i\upgeon(\bPhi_t^{-1}(\bx))$.
\end{proof}

Before writing the complete algorithm we need to make a change of
basis to express the ALE velocity in the approximation basis.  We
further assume that
\begin{equation}
\wP\upgeo \subset \wP. \label{wPgeo_subset_wP}
\end{equation}
This assumption implies that $\bP\upgeo(\calT_h^n)\subset\bP_d(\calT_h^n)$;
hence there is a sparse matrix
$\polB$, independent of $n$, such that $\psi_i\upgeon = \sum_{j\in
  s(i)} \polB_{ij} \psi_j^n$, were $s(i)$ is a sparse set of indices
for any $i\in\intset{1}{\Nglobgeo}$. We then define
\begin{equation}
\bsfW_j^n :=\sum_{\{i\st j\in
  s(i)\}} \polB_{ij} \bsfW_{i}\upgeon, \label{change_of_basis_on_W}
\end{equation}
which, owing to \eqref{Eq:Lem:TKt_varphit_bvale}, gives the following
alternative representation of $\bvale$:
\begin{equation}
\bvale(\bx,t)=\sum_{i\in\intset{1}{\Nglob}}  \bsfW_i^n  \psi_i^n
(\bPhi_t^{-1}(\bx)).
\label{change_of_basis_on_vale}
\end{equation}

\subsection{The ALE algorithm, version 2}
\label{sec:alternative}
It may look odd to some readers that in version~1 of the algorithm we
update the mass of the shape function $\psi_i^{n+1}$ by using
\eqref{alg1_def_of_fmi} instead of using $m_{i}^{n+1} =\int_{\Real^d}
\psi_{i}^{n+1}(\bx) \diff \bx$. We propose in this section an
alternative form of the algorithm that does exactly that. This
algorithm is henceforth referred to as version~2. For reasons that
will be detailed in \S\ref{Sec:SSP_implementation}, we have not been
able so far to construct an SSP extension of this algorithm that is
both conservative and invariant domain preserving, whereas the SSP
extension of version~1 is trivial.

Let $(\omega_l,\zeta_l)_{l\in\calL}$ be a quadrature such that
$\int_0^1 f(\zeta)\diff\zeta \simeq \sum_{l\in\calL} \omega_l
f(\zeta_l)$ is exact for all polynomial function $f$ of degree at most
$d-1$. We denote $t_l^n = t^n + \dt \zeta_l$.  Given the ALE field
$\bw^n\in\bP\upgeo(\calT_h^n)$, the Lagrange nodes of the mesh are
moved for each time $t_l^n$, $l\in\calL$, by using \eqref{motion_of_ai}:
\begin{equation}
\ba_i(t_l^n) =  \ba_i^n +  \dt \zeta_l \bsfW_i\upgeon.
\label{motion_of_ai_tln}
\end{equation} 
This defines the new meshes $\calT_h(t_l^n)$. This allows us to compute 
\begin{equation}
\bc_{ij}(t^n_l) := \int_{S_i^n(t^n_l)} \GRAD\varphi_j(\bx,t^n_l)
\varphi_i(\bx,t^n_l)\diff \bx,\qquad \bc_{ij}^{n} :=\sum_{l\in\calL}
\omega_l \bc_{ij}(t^n_l).
\label{def_of_cij_gauss}
\end{equation}
After constructing $\calT_h(t^{n+1})$ by setting $\ba_i^{n+1} = \ba_i^n
+ \dt \bsfW_i\upgeon$, we define the mass of $\psi_i^{n+1}$ by
\begin{equation}
m_i^{n+1} := \int_{\Real^d} \psi_{i}^{n+1}(\bx) \diff \bx.
\label{def_mass_gauss}
\end{equation} 
Then the change of basis \eqref{change_of_basis_on_W} is applied to
obtain the representation of $\bw^n$ in $\bP_d(\calT_h^n)$.  
Following \eqref{pde_ale_weak_formulation}, we 
compute $\bu_h^{n+1}$ by using the following explicit technique:
 \begin{equation}
\frac{m_i^{n+1} \bsfU_i^{n+1}-m_i^{n} \bsfU_i^n}{\dt}  
+ \sum_{j\in\calI(S_i^n)}
(\bef(\bsfU_j^n)  - \bsfU_j^n\otimes \bsfW_j^{n})\SCAL \bc_{ij}^n- d_{ij}^{n} \bsfU^n_j
 = 0,  \label{def_of_scheme_dij_gauss}
\end{equation}
where $d_{ij}^n$ is computed by using \eqref{def_of_dij}.

\begin{lemma}[Non-conservative form] \label{Lem:one_mi_gauss}
The scheme \eqref{def_of_scheme_dij_gauss} is equivalent to
\begin{align}
  m_i^{n+1} \frac{\bsfU_i^{n+1}-\bsfU_i^n}{\dt} &= \sum_{j\in
    \calI(S_i^n)}
  ((\bsfU_j^n-\bsfU_i^n)\otimes\bsfW_j^{n}-\bef(\bsfU_j^n))\SCAL
  \bc_{ij}^{n} + d_{ij}^{n} \bsfU^n_j.\label{Eq:Lem:one_mi_gauss}
\end{align}
\end{lemma}
\begin{proof}
We rewrite
\eqref{def_of_scheme_dij_gauss}
as follows:
\begin{align*}
m_i^{n+1} \frac{\bsfU_i^{n+1}-\bsfU_i^n}{\dt} + \frac{m_i^{n+1} -m_i^{n}}{\dt} \bsfU_i^n
= \sum_{j\in \calI(S_i^n)}  (\bsfU_j^n\otimes\bsfW_j^{n}-\bef(\bsfU_j^n))\SCAL \bc_{ij}^{n}
 + d_{ij}^{n} \bsfU^n_j.
\end{align*}
Then, recalling the expression \eqref{change_of_basis_on_vale} of
$\bvale(\bx,t)=\sum_{i\in\intset{1}{\Nglob}} \bsfW_i^{n}
\varphi_i(\bx,t)$, and using
Lemma~\ref{Lem:mass_transformation} with $\psi = \psi_i^n$ and
$\varphi(\bx,t) = \varphi_i(\bx,t)$ we infer that
\begin{align*}
 (m_i^{n+1} -m_i^{n})\bsfU_i^n
&= \dt \bsfU_i^n\sum_{l\in\calL} \omega_l \int_{\Real^d} \varphi_i(\bx,t^n_l)
  \sum_{j\in\calI(S_i^n)} \bsfW_j^{n} \SCAL \GRAD\varphi_j(\bx,t^n_l)
  \diff \bx \\
& = \dt \bsfU_i^n \sum_{j\in\calI(S_i^n)} \bsfW_j^{n}\SCAL\bc_{ij}^{n}
= \dt  \sum_{j\in\calI(S_i^n)} (\bsfU_i^n\otimes\bsfW_j^{n})\SCAL\bc_{ij}^{n},
\end{align*}
whence the result.
\end{proof}

\subsection{Version~1 vs. version~2 and SSP extension} 
\label{Sec:SSP_implementation}
We now give an overview of what has been done in the previous sections
by highlighting the main differences between the two versions of the
algorithm.

\begin{itemize}
\item Both versions of the algorithm use the two sets of reference elements:
we use $\{(\wK_r,\wP_r\upgeo,\wSigma_r\upgeo)\}_{1\le r\le \varpi}$ 
for the geometric mappings (see \eqref{def_geometric_transformation}), and we use
$\{(\wK_r,\wP_r,\wSigma_r)\}_{1\le r\le \varpi}$ for the approximation 
of $\bu$.
\item We assume that $\bw^n\in \bP_d(\calT_h^n)$ in version~1, whereas
  we assume that $\bw^n\in \bP\upgeo(\calT_h^n)$ in version~2. We also
  must assume that $\bP\upgeo(\calT_h^n)\subset \bP_d(\calT_h^n)$ (\ie
  $\wP\upgeo\subset \wP$, see \eqref{wPgeo_subset_wP}) in
  version~2, which is not the case for version~1; actually the space
  $\bP\upgeo(\calT_h^n)$ does not play any role in version~1.

\item Only the meshes $\calT_h^n$ and $\calT_h^{n+1}$ are
considered in version~1, whereas one must construct all the
intermediate meshes $\calT_h(t_l^n)$, $l\in\calL$, in version~2.

\item The mass of $\psi_i^{n+1}$ is updated by setting $\fm_i^{n+1} =
\fm_i^n +\dt \int_{S_i^n} \psi_i^n(\bx) \DIV\bw^n(\bx)\diff \bx$ in version~1, whereas
it is updated by setting $m_i^{n+1}=\int_{S_i^{n+1}}\psi_i^{n+1}(\bx) \diff \bx$ in version~2.
\end{itemize}

Retaining the invariant domain property (see \S\ref{Sec:Invariance})
and increasing the time accuracy can be done by using so-called Strong
Stability Preserving (SSP) time discretization methods.  The key is to
achieve higher-order accuracy in time by making convex combination of
solutions of forward Euler sub-steps. More precisely each time step of
a SSP method is decomposed into substeps that are all forward Euler
solutions, and the end of step solution is constructed as a convex
combination of the intermediate solutions; we refer to
\cite{Ferracina_Spijker_2005,Higueras_2005,Gottlieb_Ketcheson_Shu_2009}
for reviews on SPP techniques. Algorithm~\ref{alg:euler} illustrates one Euler step for
either version~1 or version~2 of the scheme. SSP techniques are useful
when combined with reasonable limitation strategies since the
resulting methods are both high-order, in time and space, and
invariant domain preserving.
\begin{algorithm}
	\caption{Euler step (version~1 and version~2)}
	\label{alg:euler}
	\begin{algorithmic}[1]
          \REQUIRE $\calT_h^0$, $\bu_h^{0}$, ($\fm^0$ or $m^0$), $\bw^0$, $\dt$  
          \STATE Compute $\tba_i^1 = \ba_i^0 + \dt \bw^0$, ($\widetilde{\fm}^1$ or $\tm^1$), $\tbu_h^1$, 
          and build new mesh $\tcalT_h^1$
          \RETURN  $\tcalT_h^1$, $\tbu_h^{1}$, ($\widetilde{\fm}^1$ or $\tm^1$)
	\end{algorithmic}
\end{algorithm}

We describe the SSP RK3 implementation of version~1 of the scheme in
Algorithm~\ref{alg:spp_rk3}. Generalizations to other SSP techniques
are left to the reader. 
\begin{algorithm}[H]
	\caption{SPP RK3, version~1}
	\label{alg:spp_rk3}
	\begin{algorithmic}[1]
	\REQUIRE  $\calT_h^0$, $\bu_h^{0}$, $\fm^0$, $t^0$
        \STATE Define the ALE velocity $\bw^0$ at $t^0$
        \STATE Call Euler step($\calT_h^0$, $\bu_h^{0}$, $\fm^0$, $\bw^0$, $\dt$, $\calT_h^1$, $\bu_h^{1}$, 
$\fm^1$)
        \STATE Define the ALE velocity $\bw^1$ at $t^{0}+\dt$
        \STATE Call Euler step($\calT_h^1$, $\bu_h^{1}$, $\fm^1$, $\bw^1$, $\dt$, $\tcalT_h^2$, $\tbu_h^{2}$, 
$\widetilde{\fm}^2$)
        \STATE Set $\ba^2=\frac34 \ba^0 + \frac14 \tba^2$, $\fm^2=\frac34 \fm^0 + \frac14 \widetilde\fm^2 $, build mesh $\calT_h^2$,
        $\bu_h^2=\frac34 \frac{\fm^0}{\fm^2}\bu_h^{0}+\frac14 \frac{\widetilde\fm^{2}}{\fm^2}\tbu_h^{2}$
        \STATE Define the ALE velocity $\bw^2$ at $t^0+\frac12\dt$
        \STATE Call Euler step($\calT_h^2$, $\bu_h^{2}$, $\fm^2$, $\bw^2$, $\dt$, $\tcalT_h^3$, $\tbu_h^{3}$, 
$\widetilde{\fm}^3$)
         \STATE Set $\ba^3=\frac13 \ba^0 + \frac23 \tba^3 $, $\fm^3=\frac13 \fm^0 + \frac23 \widetilde\fm^3$, build mesh $\calT_h^3$,
        $\bu_h^3=\frac13 \frac{\fm^0}{\fm^3}\bu_h^{0}+\frac23 \frac{\widetilde\fm^{3}}{\fm^3}\tbu_h^{3}$
        \RETURN  $\calT_h^3$, $\bu_h^{3}$, $\fm^3$, $t^1=t^0+dt$
	\end{algorithmic}
\end{algorithm}
Note that $\bu_h^2$ is a convex combination of $\bu_h^0$ and $\tbu_h^2$
since $1=\frac34 \frac{\fm_i^0}{\fm_i^2}+\frac14
\frac{\widetilde\fm_i^{2}}{\fm_i^2}$. The same observation holds for
$\bu_h^3$, \ie $\bu_h^3$ is a convex combination of $\bu_h^0$ and
$\tbu_h^3$ since $1=\frac13\frac{\fm_i^0}{\fm_i^3}+\frac23
\frac{\widetilde\fm_i^{3}}{\fm_i^3}$, for any $i\in\intset{1}{\Nglob}$.

\begin{remark}[Version~2+SSP] \label{Rem:SSP_v2}
The above properties do
not hold for version~2 of the scheme, since in general $m_i^2\ne \frac34
m_i^0+\frac14 \tm_i^{2}$ and $m_i^3\ne \frac13 m_i^0+\frac23 \tm_i^{3}$. Notice
though that it can be shown that $m_i^2= \frac34 m_i^0+\frac14 \tm_i^{2}$
and $m_i^3= \frac13 m_i^0+\frac23 \tm_i^{3}$ in one space dimension if the
ALE velocity is kept constant over the entire Runge Kutta step. So
far, we are not aware of any SSP technique for version~2 of the
algorithm (at least second-order in time) that is both conservative
and invariant domain preserving in the multidimensional case.
\end{remark}

\section{Stability analysis}\label{Sec:stability_analysis}
We establish the conservation and the invariant domain property of the
two schemes \eqref{def_of_scheme_dij} and
\eqref{def_of_scheme_dij_gauss} in this section.

\subsection{Conservation}
We first discuss the conservation properties of the two schemes.
\begin{lemma} {\em(i)} For the scheme~\eqref{def_of_scheme_dij_gauss}, the
  quantity $\int_{\Real^d} \bu_h^{n}\diff \bx$ is conserved for all
  $n\ge 0$, \ie $\int_{\Real^d} \bu_h^{n}\diff x=\int_{\Real^d}
  \bu_h^0\diff \bx$ for all $n\ge 0$. {\em(ii)} For the
  scheme~\eqref{def_of_scheme_dij}, the quantity
  $\sum_{i\in\intset{1}{\Nglob}} \fm_i^{n} \bsfU_i^{n}$ is conserved,
  \ie it is independent of $n$.
\end{lemma}
\begin{proof} We start with by proving (i).
We observe first that 
\begin{align*}
\sum_{i\in\intset{1}{\Nglob}} m_i^n \bsfU_i^{n} &=
\sum_{i\in\intset{1}{\Nglob}} \bsfU_i^{n}\sum_{j\in\intset{1}{\Nglob}} 
\int_{\Real^d} \psi_i^n(\bx)\psi_j^n(\bx) \diff \bx\\
&=  \sum_{j\in\intset{1}{\Nglob}} \int_{\Real^d}\sum_{i\in\intset{1}{\Nglob}}  \bsfU_i^{n}
\psi_i^n(\bx)\psi_j^n(\bx) \diff \bx = 
\sum_{j\in\intset{1}{\Nglob}} \int_{\Real^d} \bu_h^n(\bx)\psi_j^n(\bx) \diff \bx.
\end{align*}
Then the partition of unity property gives
$\sum_{i\in\intset{1}{\Nglob}} m_i^n \bsfU_i^{n} = \int_{\Real^d}
\bu_h^n\diff \bx$. We now sum over the index $i$ in
\eqref{def_of_scheme_dij} and we use again the partition of unity
property to infer that
\begin{multline*}
\frac{\int_{\Real^d}\bu_h^{n+1}\diff \bx - \int_{\Real^d}\bu_h^n\diff
  x}{\dt}- \sum_{j\in \calI(S_i^n)} \bigg(\sum_{i\in\intset{1}{\Nglob}} d_{ij}^{n}\bigg) \bsfU^n_j \\
 + \sum_{l\in\calL} \omega_l
\int_{\Real^d} \DIV\bigg(\sum_{j\in\intset{1}{\Nglob}}
(\bef(\bsfU_j^n)  - \bsfU_j^n\otimes \bsfW_j^{n})\varphi_j(\bx,t^n_l)\bigg)\diff \bx 
  = 0.
\end{multline*} 
The boundary conditions and the structure assumptions on $d_{ij}^{n}$, see
\eqref{introduction_dij},
imply the desired result. The proof of (ii) follows the same lines.
\end{proof}

\subsection{Invariant domain property} \label{Sec:Invariance}
We can now prove a result somewhat similar in spirit to Thm~5.1 from
\cite{Farhat_Geuzaine_Grandmont_2001}, although the present result is
more general since it applies to any hyperbolic system.

We start with version~2 of the scheme by defining the local minimum
mesh size $\hmin_{ij}(t)$ associated with an ordered pair of shape
functions $(\varphi_i(\cdot,t),\varphi_j(\cdot,t))$ at time $t$ as follows:
$\hmin_{ij}(t):=
\frac{1}{\|\|\GRAD\varphi_j\|_{\ell^2}\|_{L^\infty(S_{ij}(t))}},$
where $S_{ij}(t)= S_i(t)\cap S_j(t)$.  We also define $\hmin_{i}(t) =
\min_{j\in\calI(S_i(t))} \hmin_{ij}(t)$.  Given a time $t^n$, we
define a local minimum mesh size $\hmin_i^n$ and a local mesh structure
parameter $\kappa_i^n$ by
\begin{equation} \label{eq:local_mesh_gauss} \hmin_i^n:=
  \min_{l\in\calL} \hmin_{i}(t_l^n),\qquad \kappa_i^n :=
  \frac{\sum_{i\not= j\in\calI(S_i^n)}\sum_{l\in \calL} \omega_l
    \int_{S_{ij}(t_l^n)}\varphi_i(\bx,t_l^n) \diff \bx}{\sum_{l\in
      \calL} \omega_l m_i(t_l^n)}.
\end{equation}
For version~1 of the algorithm we set 
\begin{equation} \label{eq:local_mesh}
\hmin_i^n:= \hmin_{i}(t^n),\qquad
  \kappa_i^n := \frac{\sum_{i\not= j\in\calI(S_i^n)}
\int_{S_{ij}^n}\psi_i^n(\bx) \diff \bx}{\int_{S_i^n}\psi_i^n(\bx) \diff \bx}.
\end{equation} 
Note that the upper estimate $\kappa_i^n \le
\max_{j\in\intset{1}{\Nglob}} \text{card}(\calI(S_j(0)))-1$ implies
that $\kappa_i^n$ is uniformly bounded with respect to $n$ and $i$ for
both algorithms.

\begin{theorem}[Local invariance] \label{Thm:invariant_domain} Let $n\ge 0$,
  and let $i\in\intset{1}{\Nglob}$.
  $\lambda_{i,\max}^n:= \max_{j\in \calI(S_i^n)}
  (\lambda_{\max}(\bg_{j}^n,\bn_{ij}^n,\bsfU_i^n,\bsfU_j^n),$ 
$\lambda_{\max}(\bg_{i}^n,\bn_{ji}^n,\bsfU_j^n,\bsfU_i^n))$. Depending on the version of the algorithm,
version~2 or version~1 respectively, assume that $\dt$ is such that
\begin{equation}
  2\dt\frac{\lambda_{i,\max}^n}{\hmin_i^n} \kappa_i^n 
 \frac{\sum_{l\in \calL} \omega_l m_i(t_l^n)}{m_i^{n+1}}  \le 1,\quad\text{or}\quad
2\dt\frac{\lambda_{i,\max}^n}{\hmin_i^n} \kappa_i^n 
 \frac{\fm_i^n}{\fm_i^{n+1}}  \le 1.
\label{CFL_assumption}
\end{equation}
Let $B\subset \calA_\bef$ be a convex invariant set for the flux $\bef$ such that
$\{\bsfU_j^n\st j\in \calI(S_i^n)\}\subset B$, then
$\bsfU_i^{n+1} \in B$.
\end{theorem}
\begin{proof} We do the proof for version~2 of the algorithm.  The
  proof for version~1 is similar.  Let $i\in\intset{1}{\Nglob}$ and
  invoke \eqref{Eq:Lem:one_mi_gauss} from Lemma~\ref{Lem:one_mi_gauss} 
(or \eqref{Eq:Lem:one_mi} from Lemma~\ref{Lem:one_mi}  for version~1)
to express
  $\bsfU_i^{n+1}$ into the following from
\begin{align*}
  \bsfU_i^{n+1} &= \bsfU_i^n +\frac{\dt}{m_i^{n+1}}\sum_{j\in
    \calI(S_i^n)}
  ((\bsfU_j^n-\bsfU_i^n)\otimes\bsfW_j^{n}-\bef(\bsfU_j^n))\SCAL
  \bc_{ij}^{n} + d_{ij}^{n} \bsfU^n_j.
\end{align*}
Since the partition of unity property implies that 
$\sum_{j\in\calI(S_i^n)} \bc_{ij}^{n} =0$ and we have
$\sum_{j\in\calI(S_i^n)} d_{ij}^{n}=0$ from \eqref{introduction_dij},
we can rewrite the above equation as follows:
\begin{align*}
\bsfU_i^{n+1} &= \bsfU_i^n+ \sum_{j\in \calI(S_i^n)} d_{ij}^{n} (\bsfU^n_i + \bsfU^n_j)\\
&\hspace{0.1\linewidth} +\frac{\dt}{m_i^{n+1}}\sum_{j\in \calI(S_i^n)} ((\bsfU_j^n-\bsfU_i^n)\otimes\bsfW_j^{n} 
+\bef(\bsfU_i^n)-\bef(\bsfU_j^n))\SCAL \bc_{ij}^{n} \\
&= \bsfU_i^n\left(1+2d_{ii}^{n}\frac{\dt}{m_i^{n+1}}\right)+ \sum_{i\ne j\in \calI(S_i^n)}d_{ij}^{n} (\bsfU^n_i + \bsfU^n_j)\\
&+\frac{\dt}{m_i^{n+1}}\!\!\sum_{i\ne j\in \calI(S_i^n)}\!\!((\bsfU_j^n-\bsfU_i^n)\otimes\bsfW_j^{n} 
+\bef(\bsfU_i^n)-\bef(\bsfU_j^n))\SCAL \bc_{ij}^{n}.
\end{align*}
Recall that
$\bn_{ij}^{n}:=\bc_{ij}^{n}/\|\bc_{ij}^{n}\|_{\ell^2} \in
S^{d-1}(\bzero,1)$, and let us 
introduced the auxiliary state $\overline\bsfU_{ij}^{n+1}$
defined by
\[
\overline\bsfU_{ij}^{n+1}=(\bef(\bsfU_i^n)-\bef(\bsfU_j^n)
-(\bsfU_i^n-\bsfU_j^n)\otimes\bsfW_j^{n})\SCAL \bn_{ij}^{n}\frac{\|\bc_{ij}^{n}\|_{\ell^2}}{2d_{ij}^{n}}
 + \frac12 (\bsfU^n_i + \bsfU^n_j).
\]
Then, provided we establish that $1-\sum_{i\ne j\in
  \calI(S_i^n)} 2d_{ij}^{n}\frac{\dt}{m_i^{n+1}}\ge 0$, we have proved
that $\bsfU_i^{n+1}$ is a convex combination of $\bsfU_i^{n}$ and
$(\overline\bsfU_{ij}^{n+1})_{i\ne j\in \calI(S_i^n)}$:
\begin{equation}
\bsfU_i^{n+1} =
\bsfU_i^n\bigg(1-\sum_{i\ne j\in \calI(S_i^n)}
  2d_{ij}^{n}\frac{\dt}{m_i^{n+1}}\bigg)
+\frac{\dt}{m_i^{n+1}}\sum_{i\ne j\in \calI(S_i^n)} 2 d_{ij}^{n} \overline\bsfU_{ij}^{n+1}.
\label{convex_combination_Ui}
\end{equation}
Let us now consider the Riemann problem \eqref{one_D_Riemann_problem}.
Let $\bv(\bg_j^n,\bn_{ij}^n,\bsfU_i^n,\bsfU_j^n)$ be the solution to
\eqref{one_D_Riemann_problem} with $\bg_j^n(\bv) := \bef(\bv) -
\bv\otimes \bsfW_j^n$.  Let
$\lambda_{\max}(\bg_{j}^n,\bn_{ij}^n,\bsfU_i^n,\bsfU_j^n)$ be the
fastest wave speed in \eqref{one_D_Riemann_problem}, see
\eqref{lambda_max_shifted}.  Using the notation of
Lemma~\ref{Lem:elementary_Riemann_pb}, we then observe that
\[
\overline\bsfU_{ij}^{n+1} =
\overline\bv(t,\bg_j^n,\bn_{ij}^{n},\bsfU_i^n,\bsfU_j^n)
\] 
with $t=\frac{\|\bc_{ij}^{n}\|_{\ell^2}}{2d_{ij}^{n}}$, provided
$t\lambda_{\max}(\bg_j^n,\bn_{ij}^{n},\bsfU_i^n,\bsfU_j^n)\le
\frac12$.  Note that the definition of $d_{ij}^{n}$,
\eqref{def_of_dij}, implies that the condition
$t\lambda_{\max}(\bg_j^n,\bn_{ij}^{n},\bsfU_i^n,\bsfU_i^n)\le \frac12$
is satisfied. Since $B$ is an invariant set for the flux $\bef$, by
Lemma~\ref{Lem:same_invariant_sets}, $B$ is also an invariant set for
the flux $\bg_j^n$.  Since, in addition, $B$ contains the data
$(\bsfU_i^n,\bsfU_j^n)$, we conclude that $\overline\bsfU_{ij}^{n+1}=
\overline\bv(t,\bg_j^n,\bn_{ij}^{n},\bsfU_i^n,\bsfU_j^n)\in B$; see
Remark~\ref{Rem:Riemann_fan_average}.  In conclusion,
$\bsfU_i^{n+1}\in B$ since $\bsfU_i^{n+1}$ is a convex combination of
objects in $B$.

Setting $\calY:=\sum_{i\not= j\in\calI(S_{ij}(t_l^n))}\frac{2\dt
  d_{ij}^n}{m_i^{n+1}}$, there remains to establish that $1-\calY\ge 0$ to complete the proof for
version~2 of the algorithm.  Note first that
\begin{align*}
  \|\bc_{ij}(t_l^n)\|_{\ell^2} & \le \int_{S_{ij}(t_l^n)}
  \|\GRAD\varphi_j(\bx,t^n_l)\|_{\ell^2} \varphi_i(\bx,t^n_l)\diff \bx
  \le \hmin_{ij}^{-1}(t_l^n) \int_{S_{ij}(t_l^n)}\varphi_i(\bx,t_l^n)
  \diff \bx.
\end{align*}
Notice that $\bc_{ij}(t_l^n)=-\bc_{ji}(t_l^n)$ because there are no
boundary conditions (\ie we solve the Cauchy problem in $\Real^d$, or the domain is periodic);
hence $\|\bc_{ji}(t_l^n)\|_{\ell^2} =\|\bc_{ij}(t_l^n)\|_{\ell^2}$.
Recalling the definition of $d_{ij}^n$, we have
\begin{align*}
\calY
& \le 2\dt\frac{\lambda_{i,\max}^n}{\hmin_i^n}  
\frac{\sum_{l\in \calL} \omega_l m_i(t_l^n)}{m_i^{n+1}}
\frac{\sum_{i\not= j\in\calI(S_{ij}(t_l^n))}\sum_{l\in \calL} \omega_l
\int_{S_{ij}(t_l^n)}\varphi_i(\bx,t_l^n) \diff \bx}{\sum_{l\in \calL} \omega_l m_i(t_l^n)}
\\
&\le 2\dt\frac{\lambda_{\max}^n}{\hmin_i^n}  
\frac{\sum_{l\in \calL} \omega_l m_i(t_l^n)}{m_i^{n+1}} \kappa^n_i \le 1,
\end{align*}
which is the desired result. The proof of the CFL condition for
version~1 of the algorithm follows the same lines.  This concludes the
proof.
\end{proof}

\begin{corollary}
  Let $n\in \polN$.  Assume that $\dt$ is small enough so that the CFL
  condition \eqref{CFL_assumption} holds for all
  $i\in\intset{1}{\Nglob}$. Let $B\subset \calA_\bef$ be a convex
  invariant set.  Assume that $\{\bsfU_i^n \st
  i\in\intset{1}{\Nglob}\}\subset B$. Then {\em (i)} $\{\bsfU_i^{n+1}
  \st i\in\intset{1}{\Nglob}\}\subset B$ {\em (ii)} $\bu_h^n\in B$ and
  $\bu_h^{n+1}\in B$.
\end{corollary}
\begin{proof}
  The statement (i) is a direct consequence of
  Theorem~\ref{Thm:invariant_domain}.  The statement (ii) is a
  consequence of \eqref{mesh_convexity_assumption} from
  Lemma~\ref{Lem:convexity}.
\end{proof}

\begin{corollary} \label{Cor:invariance} Let $B\subset \calA_\bef$ be a convex
  invariant set containing the initial data $\bu_0$.  Assume that
  $\{\bsfU_i^0 \st i\in\intset{1}{\Nglob}\}\subset B$. Let $N\in
  \polN$.  Assume that $\dt$ is small enough so that the CFL condition
  \eqref{CFL_assumption} holds for all $i\in\intset{1}{\Nglob}$ and
  all $n\in\intset{0}{N}$.  Then $\{\bsfU_i^n \st
  i\in\intset{1}{\Nglob}\}\subset B$ and $\bu_h^{n}\in B$ for all
  $n\in\intset{0}{N+1}$.
\end{corollary}

\begin{remark}[Construction of $\bu_h^0$]
  Let $B\subset \calA_\bef$ be a convex invariant set containing the
  initial data $\bu_0$. If $\bP_m(\calT_h^0)$ is composed of piecewise
  Lagrange elements, then defining $\bu_h^0$ to be the Lagrange
  interpolant of $\bu_0$, we have $\{\bsfU_i^0 \st
  i\in\intset{1}{\Nglob}\}\subset B$. Similarly if $\bP_m(\calT_h^0)$ is
  composed of Bernstein finite elements of degree two and higher, then
  defining $\bu_h^0$ to be the Bernstein interpolant of $\bu_0$ we
  have $\{\bsfU_i^0 \st i\in\intset{1}{\Nglob}\}\subset B$; see
  \cite[Eq. (2.72)]{Lai_Schumaker_2007}. Note that the approximation
  of $\bu_0$ is only second-order accurate in this case independently
  of the polynomial degree of the Bernstein polynomials; see
  \citep[Thm.~2.45]{Lai_Schumaker_2007}. In both cases the assumptions
  of Corollary~\ref{Cor:invariance} hold.
\end{remark}

\subsection{Discrete Geometric Conservation Law} Both
the scheme~\eqref{def_of_scheme_dij} and the scheme~\eqref{def_of_scheme_dij_gauss}
satisfy the so-called Discrete
Geometric Conservation Law (DGCL), \ie they preserve constant states.
\begin{corollary}[DGCL]  \label{Cor:DGCL}
The schemes~\eqref{def_of_scheme_dij} and \eqref{def_of_scheme_dij_gauss} 
preserve constant states.
In particular if $\bsfU_j^n=\bsfU_i^n$ for all $j\in\calI(S_i^n)$, then
$\bsfU_i^{n+1}=\bsfU_i^n$.
\end{corollary}
\begin{proof}
  The partition of unity property implies that
  $\sum_{j\in\calI(S_i^n)} \bc_{ij}^{n} =0$ for both schemes. Moreover,
the definition $d_{ij}^{n}$, which is common for both schemes, implies that
  $\sum_{j\in\calI(S_i^n)} d_{ij}^{n}=0$
  (see \eqref{introduction_dij}). For the scheme~\eqref{def_of_scheme_dij_gauss},
Lemma~\ref{Lem:one_mi_gauss}, which we
  recall is a consequence of Lemma~\ref{Lem:mass_transformation}, implies that
\begin{align*}
\bsfU_i^{n+1} = \bsfU_i^n &+ d_{ij}^{n} (\bsfU^n_j -\bsfU^n_i) \\
&+\frac{\dt}{m_i^{n+1}}\sum_{j\in \calI(S_i^n)} ((\bsfU_j^n-\bsfU_i^n)\otimes\bsfW_j^{n} 
+\bef(\bsfU_i^n)-\bef(\bsfU_j^n))\SCAL \bc_{ij}^{n} .
\end{align*}
For the scheme~\eqref{def_of_scheme_dij}, Lemma~\ref{Lem:one_mi} implies that
\begin{align*}
\bsfU_i^{n+1} = \bsfU_i^n &+ d_{ij}^{n} (\bsfU^n_j -\bsfU^n_i) \\
&+\frac{\dt}{\fm_i^{n+1}}\sum_{j\in \calI(S_i^n)} ((\bsfU_j^n-\bsfU_i^n)\otimes\bsfW_j^{n} 
+\bef(\bsfU_i^n)-\bef(\bsfU_j^n))\SCAL \bc_{ij}^{n} .
\end{align*}
It is now clear that if $\bsfU_j^n=\bsfU_i^n$ for all $j\in\calI(S_i^n)$, then
$\bsfU_i^{n+1}=\bsfU_i^n$.
\end{proof}

\begin{remark}[DGCL] Note that although the DGCL seems to be
  given some importance in the literature, Corollary~\ref{Cor:DGCL} has no
    particular significance. For scheme~2, it is a direct consequence
    of Lemma~\ref{Lem:mass_transformation} which is invoked to rewrite
    the scheme \eqref{def_of_scheme_dij_gauss} from the conservative
    form to the equivalent nonconservative form
    \eqref{Eq:Lem:one_mi_gauss}.  For scheme~1, it is a direct
    consequence of the definition of the mass update
    \eqref{alg1_def_of_fmi} which is invoked to rewrite the scheme
    \eqref{def_of_scheme_dij} from the conservative form to the
    equivalent nonconservative form \eqref{Eq:Lem:one_mi}.  The
    nonconservative form of both schemes is essential to prove the
    invariant domain property. In other words, {\em the DGCL is
    just a consequence of the equivalence of the discrete conservative
    and nonconservative formulations.}
\end{remark}

\subsection{Discrete entropy inequality} 
In this section we prove a discrete entropy inequality which is
consistent with the inequality stated in
Lemma~\ref{Lem:ALE_hyperbolic_system_entropy}.

\begin{theorem} \label{Thm:disrete_entropy_inequality}
  Let $(\eta,\bq)$ be an entropy pair for
  \eqref{def:hyperbolic_system}. Let $n\in\polN$ and
  $i\in\intset{1}{\Nglob}$. Assume that all the assumptions of
  Theorem~\ref{Thm:invariant_domain} hold. Then the following discrete
  entropy inequality holds for scheme~1:
\begin{multline}
\frac{1}{\dt}\big(\fm_i^{n+1}\eta(\bsfU_i^{n+1}) - \fm_i^{n}\eta(\bsfU_i^{n})\big) 
\le -\sum_{j\in \calI(S_i^n)} d_{ij}^n \eta(\bsfU_{j}^{n}) \\
 - \int_{\Real^d} \DIV\bigg(\sum_{j\in\calI(S_i^n)}
(\bq(\bsfU_j^{n}) -  \eta(\bsfU_j^{n})\bsfW_j^{n})\psi_j^n(\bx)\bigg)
\psi_i^n(\bx)\diff \bx
\end{multline}
and the following holds for scheme~2:
\begin{multline}
\frac{1}{\dt}\big(m_i^{n+1}\eta(\bsfU_i^{n+1}) - m_i^{n}\eta(\bsfU_i^{n})\big) 
\le -\sum_{j\in \calI(S_i^n)} d_{ij}^n \eta(\bsfU_{j}^{n}) \\
 - \sum_{l\in\calL} \omega_l \int_{\Real^d} \DIV\bigg(\sum_{j\in\calI(S_i^n)}
(\bq(\bsfU_j^{n}) -  \eta(\bsfU_j^{n})\bsfW_j^{n})\varphi_j(\bx,t_l^n)\bigg)
\varphi_i(\bx,t_l^n)\diff \bx
\end{multline}
\end{theorem}
\begin{proof}
We only do the proof for scheme~2. The proof for scheme~1 is similar.  
  Let $(\eta,\bq)$ be an entropy pair for the hyperbolic
  system~\eqref{def:hyperbolic_system}. Let $i\in\intset{1}{\Nglob}$
  and let $n\in\polN$.  Then using~\eqref{convex_combination_Ui}, the
  CFL condition and the convexity of $\eta$, we have
\begin{align*}
\eta(\bsfU_i^{n+1}) \le
\eta(\bsfU_i^n)\bigg(1-\sum_{i\ne j\in \calI(S_i^n)}
  2d_{ij}^{n}\frac{\dt}{m_i^{n+1}}\bigg)
+\frac{\dt}{m_i^{n+1}}\sum_{i\ne j\in \calI(S_i^n)} 2 d_{ij}^{n} \eta(\overline\bsfU_{ij}^{n+1}).
\end{align*}
This can also be rewritten as follows:
\begin{align*}
\frac{m_i^{n+1}}{\dt}\big(\eta(\bsfU_i^{n+1}) - \eta(\bsfU_i^{n})\big) \le
\sum_{i\ne j\in \calI(S_i^n)} 2 d_{ij}^{n} (\eta(\overline\bsfU_{ij}^{n+1})-\eta(\bsfU_i^n)).
\end{align*}
Owing to \eqref{entropy_elementary_Riemann_pb} from
Lemma~\ref{Lem:elementary_Riemann_pb}, and recalling that 
the entropy flux of the Riemann problem~\eqref{one_D_Riemann_problem} is
$(\bq(v) - \eta(v) \bsfW_j^n)\SCAL\bn_{ij}^n$ 
we infer that
\[
\eta(\overline\bsfU_{ij}^{n+1}) \le \tfrac12(\eta(\bsfU_{i}^{n})+\eta(\bsfU_{j}^{n}))
-t \big(\bq(\bsfU_{j}^{n}) - \eta(\bsfU_{j}^{n}) \bsfW_j^n
- \bq(\bsfU_{i}^{n}) + \eta(\bsfU_{i}^{n}) \bsfW_j^n\big)\SCAL\bn_{ij}^n
\]
with $t=\|\bc_{ij}^{n}\|_{\ell^2}/2 d_{ij}^n$. 
Inserting this inequality in the first one, we have
\begin{multline*}
\frac{m_i^{n+1}}{\dt}\big(\eta(\bsfU_i^{n+1}) - \eta(\bsfU_i^{n})\big) 
\le  \sum_{j\in \calI(S_i^n)}  d_{ij}^{n} (\eta(\bsfU_j^n)-\eta(\bsfU_i^n))\\
- \sum_{j\in \calI(S_i^n)} \|\bc_{ij}^{n}\|_{\ell^2} \big(\bq(\bsfU_{j}^{n})- \bq(\bsfU_{i}^{n}) - (\eta(\bsfU_{j}^{n}) 
 - \eta(\bsfU_{i}^{n})) \bsfW_j^n\big)\SCAL\bn_{ij}^n.
\end{multline*}
By proceeding as in the proof of Lemma~\ref{Lem:one_mi_gauss},
we observe that
\[
\frac{m_i^{n+1} -m_i^n}{\dt}
= \sum_{l\in\calL} \omega_l \int_{\Real^d} \varphi_i(\bx,t^n_l)
  \sum_{j\in\calI(S_i^n)} \bsfW_j^{n} \SCAL \GRAD\varphi_j(\bx,t^n_l)
  \diff \bx 
 =  \sum_{j\in\calI(S_i^n)} \bsfW_j^{n}\SCAL\bc_{ij}^{n}.
\]
Then using that $\|\bc_{ij}^{n}\|_{\ell^2}\bn_{ij}^n = \bc_{ij}^{n}$, we conclude that
\begin{multline*}
\frac{1}{\dt}\big(m_i^{n+1}\eta(\bsfU_i^{n+1}) - m_i^{n}\eta(\bsfU_i^{n})\big) 
\le -\sum_{j\in \calI(S_i^n)} d_{ij} \eta(\bsfU_{j}^{n}) \\
 - \sum_{l\in\calL} \omega_l \int_{\Real^d} \DIV\bigg(\sum_{j\in\calI(S_i^n)}
(\bq(\bsfU_j^{n}) -  \eta(\bsfU_j^{n})\bsfW_j^{n})\varphi_j(\bx,t_l^n)\bigg)
\varphi_i(\bx,t_l^n)\diff \bx
\end{multline*}
This concludes the proof.
\end{proof}

\section{Numerical tests} \label{Sec:numerical_tests}
In this section, we numerically illustrate the performance of
scheme~1 using SSP RK3.  All the tests reported below
have also been done with version~2 and we have observed that the
method works as advertised when used with Euler time stepping, but we do
not show the results for brevity. As expected from
Remark~\ref{Rem:SSP_v2}, we have indeed observed very small violations
of the invariant domain (maximum principle in the scalar case) when version~2 is combined with SSP RK3.

All the tests have been done with two different codes. One code is
written in F95 and uses $\polP_1$ Lagrange elements on triangles.  The
other code is based on deal.ii~\citep{BangerthHartmannKanschat2007},
is written in C++ and uses $\polQ_1$ Lagrange elements on
quadrangles. The mesh composed of triangles is obtained by dividing
all the quadrangles into two triangles. The same numbers of degrees of
freedom are used for both codes.

\subsection{Analytical scalar-valued solution}
To test the convergence property of the SSP RK3 version of scheme~1, as described in
Algorithm~\ref{alg:spp_rk3}, we solve the linear transport equation in the domain 
$\Dom^0=(0,1)^2$:
\begin{equation}
    \label{test_rotation_jlg}
    \partial_t u + \DIV (\bbetaa u) = 0,\quad
    u_0(\bx) = x_1+x_2,
\end{equation} 
where $\bbetaa = (\sin(\pi x_1)\cos(\pi x_2) \cos(2\pi t),-\cos(\pi
x_1)\sin(\pi x_2) \cos(2\pi t))\tr$.  
In both codes the ALE velocity is chosen by setting
$\bsfW_i^n=\bbetaa(\ba_i^n)$, \ie $\bw_h^n$ is the Lagrange
interpolant of $\bbetaa$ on $\calT_h^n$. Notice that there is no issue with
boundary condition since $\bbetaa\SCAL \bn_{|\front^0}=0$.
\begin{table}[h]
\begin{center}
\caption{Rotation problem \eqref{test_rotation_jlg} with Lagrangian formulation, CFL=1.0}
\label{tab:rotation_cfl10_jlg}
\begin{tabular}{|r|c|c|c|c||c|c|c|c|} \hline
 &  \multicolumn{4}{|c||}{Without viscosity} & \multicolumn{4}{|c|}{With viscosity} \\ \hline
\!\!\# dofs & \multicolumn{2}{|c|}{$\polQ_1$, $L^1$-norm}& \multicolumn{2}{|c||}{$\polP_1$, $L^1$-norm}
& \multicolumn{2}{|c|}{$\polQ_1$, $L^1$-norm}& \multicolumn{2}{|c|}{$\polP_1$, $L^1$-norm} \\ \hline
81   &  6.46E-04 & -   & 1.76E-03 & -    &1.31E-02 & -   & 1.13E-02 & - \\ \hline
289  &  1.16E-04 & 2.48& 2.46E-04 & 2.85 &4.28E-03 & 1.61& 3.63E-03 &1.64 \\ \hline
1089 &  1.41E-05 & 3.03& 3.23E-05 & 2.93 &1.23E-03 & 1.80& 1.04E-03 &1.80 \\ \hline
4225 &  1.76E-06 & 3.01& 4.20E-06 & 2.94 &3.29E-04 & 1.90& 2.78E-04 &1.90 \\ \hline
16641&  2.26E-07 & 2.96& 5.76E-07 &	2.87 &8.50E-05 & 1.95& 7.19E-05 &1.95 \\ \hline
66049&  2.82E-08 & 3.00& 9.57E-08 & 2.59 &2.16E-05 & 1.97& 1.83E-05 &1.98 \\ \hline
\end{tabular}
\end{center}
\end{table}
We first test the accuracy in time of the algorithm by setting
$d_{ij}^n=0$, \ie the viscosity is removed.  We report the error
measured in the $L^1$-norm at time $t=0.5$ in the left part of
Table~\ref{tab:rotation_cfl10_jlg}. The computations are done with $CFL=1$.
The third-order convergence in
time is confirmed. Note that there is no space error due to the
particular choice of the ALE velocity and the initial data. 
 
In the second test we put back the viscosity $d_{ij}^n$. Notice that
the particular choice of the ALE velocity implies that
$\lambda_{\max}(\bg_{j}^n,\bn_{ij}^n,\bsfU_i^n,\bsfU_j^n)=|(\bbetaa^n_i-\bbetaa^n_j)\cdot
\bn^n_{ij}|$; hence the viscosity is second-order in space instead of
being first-order. This phenomenon makes the algorithm second-order in
space (in addition to being conservative and maximum principle
preserving). The error in the
$L^1$-norm at time $t=0.5$ is shown in the right part of   
Table~\ref{tab:rotation_cfl10_jlg}. 

\subsection{Nonlinear scalar conservation equations}
We now test scheme~1 on nonlinear scalar conservation equations.
\subsubsection{Definition of the ALE velocity}
In nonlinear conservation equations, solutions may develop shocks in finite
time.  In this case, using the purely Lagrangian velocity leads to a
breakdown on the method in finite time. The breakdown manifests itself
by a time step that goes to zero as the current time approaches the
time of formation of the shock. One way to avoid this breakdown is to
use an ALE velocity that is a modified version of the Lagrangian velocity.

Many techniques have been proposed in the literature to construct an
ALE velocity.  For instance, in \cite{Gastaldi01}, the ALE velocity is
obtained by modeling the deformation of the domain as an ``elastic''
solid, see \cite[Eq. (4.5)-(4.6)]{Gastaldi01}. In \cite{YM05}, several
mesh moving strategies are mentioned, including tension spring
analogy, torsion spring analogy, truss analogy and linear elasticity
analogy. In \cite[Eq. (7)]{WBG05}, an elliptic problem is used to
construct an ALE velocity for the Euler equations. The purpose of the
present paper is not to design an optimal ALE velocity but to propose
an algorithm that is conservative and invariant domain preserving for
any reasonable ALE velocity. We now propose an algorithm to compute an
ALE velocity based on ideas from \cite{Loubere2010}. The only purpose
of this algorithm is to be able to run the nonlinear simulations of
\S\ref{Sec:Burgers} and \S\ref{Sec:KPP} past the time of formation of
shocks. We refer the reader to the abundant ALE literature to design
other ALE velocities that better suit the reader's goals.

We first deform the mesh by using the Lagrangian motion, \ie we set
$\ba^{n+1}\idlag = \ba^{n}_i + \dt \GRAD_u\bef(\bsfU_i^n)$; we recall
that $\bsfU_i^n\in\Real$ and $\GRAD_u\bef(\bsfU_i^n)\in \Real^d$ for
scalar equations. Then, given $L\in \polN\setminus\{0\}$, we define a
smooth version of the Lagrangian mesh by smoothing the position of
the geometric Lagrange nodes as follows:
\begin{equation}
\left\{
\begin{aligned}
    \ba^{n+1,0}_{i} :=& \ba^{n+1}\idlag, \ i\in\intset{1}{\Nglob}\\
      \bigg(\ba^{n+1,l}_{i} :=& \frac{1}{|\calI(\calS_i)|-1}\sum_{i\neq j\in \calI(\calS_i)}\ba^{n+1,l-1}_{j},\
i\in\intset{1}{\Nglob}\bigg), \ l\in\intset{1}{L}\\
    \ba^{n+1}\idsmooth :=& \ba^{n+1,L}_{i},\ i\in\intset{1}{\Nglob}.
\end{aligned}\right. \label{mesh_motion}
\end{equation}
Finally, the actual ALE motion is defined by  
\begin{equation*}
    \ba_i^{n+1} = \omega \ba^{n+1}\idlag + (1-\omega)\ba^{n+1}\idsmooth, \quad i\in\intset{1}{\Nglob}
\end{equation*}
where $\omega$ is a user-defined constant. In all our computations, we
use $\omega=0.9$ and $L=2$.
The above method is similar to that used in \cite{Loubere2010}. As 
mentioned in \citep{Loubere2010}, a more advanced method consists of choosing
$\omega$ pointwise by using the right Cauchy-Green strain tensor. We
have not implemented this version of the method since the purpose of
the tests in the next sections is just to show that the
present method works as it should for any reasonable ALE velocity. 
The objective of this work is not to construct a sophisticated
algorithm for the ALE velocity.

\subsubsection{Burgers equation}
\label{Sec:Burgers}
We consider the inviscid Burgers equation in two space dimensions
\begin{equation}
\partial_t u + \DIV( \tfrac12 u^2 \bbetaa) =0, \quad
u_0(\bx) = \mathds{1}_S,
\label{Burgers_equation}
\end{equation}
where $\bbetaa=(1,1)\tr$, $S$ is the unit square $(0,1)^2$, and $\mathds{1}_E$ denotes the
characteristic function of the set $E\subset \Real^d$. The solution to this problem at time
$t>0$ and at $\bx=(x_1,x_2)$ is as follows. Assume first that $x_2\le
x_1$, then define $\alpha=x_1-x_2$. Let $\alpha_0 = 1- \frac{t}{2}$. There are three cases.
If $\alpha>1$, then $u(x_1,x_2,t) =0$.
\begin{align}\text{If $\alpha\le \alpha_0$, then}\quad
 &u(x_1,x_2,t)=
\begin{cases}
\frac{x_2}{t}&  \text{if $0\le x_2 < t$} \\
1  &  \text{if $t \le x_2 < \frac{t}{2} + 1-\alpha$} \\
0 & \text{otherwise.}
\end{cases}\\
\text{If $\alpha_0<\alpha\le 1$, then}\quad 
&u(x_1,x_2,t)=\begin{cases}
\frac{x_2}{t}&  \text{if $0\le x_2 < \sqrt{2t(1-\alpha)}$} \\
0 & \text{otherwise.}
\end{cases}
\end{align}
If $x_2> x_1$, then $u(x_1,x_2,t) := u(x_2,x_1,t)$.
The computation are done up to $T=1$ in the initial computational domain
$\Dom^0 = (-0.25,1.75)^2$. The boundary of $\Dom^n$ does not move in
the time interval $(0,1)$, \ie $\front^0=\front^n$ for any $n\ge 0$.   The results of the convergence tests
are reported in Table \ref{tab:ale_brugers_new_q1_v1_jlg}.
The solution computed on a $128\CROSS 128$ 
mesh and the mesh at $T=1$ are shown in Figure
\ref{fig:ale_brugers_new_q1_v1_jlg}. 
\begin{table}[h]
\begin{center}
    \caption{Burgers equation, convergence tests, $CFL=0.1$}
    \label{tab:ale_brugers_new_q1_v1_jlg}
\begin{tabular}{|r|c|c|c|c||c|c|c|c|} \hline
 &  \multicolumn{4}{|c||}{$\polQ_1$} & \multicolumn{4}{|c|}{$\polP_1$} \\ \hline
\!\!\# dofs & 
\multicolumn{2}{|c|}{$L^2$-error} & 
\multicolumn{2}{|c||}{$L^1$-error} &
\multicolumn{2}{|c|}{$L^2$-error} & 
\multicolumn{2}{|c|}{$L^1$-error}
\\ \hline
81       & 5.79E-01 & -      & 6.00E-01 & - & 5.80E-01& -& 6.17E-01 & -\\ \hline
289     & 4.20E-01 & 0.46 & 3.88E-01 & 0.63& 4.43E-01&0.39 & 4.68E-01&0.40\\ \hline
1089   & 2.96E-01 & 0.51 & 2.32E-01 & 0.74& 3.12E-01& 0.51& 2.86E-01& 0.71\\ \hline
4225   & 2.14E-01 & 0.47 & 1.32E-01 & 0.82& 2.17E-01&0.53 &1.55E-01 & 0.88\\ \hline
16641 & 1.56E-02 & 0.45 & 7.40E-02 & 0.83& 1.23E-01
&0.82 & 7.57E-02 & 1.04\\ \hline
\end{tabular}
\end{center}
\end{table}

\begin{figure}[H]
\centering{
\includegraphics[width=0.24\textwidth,bb=87 87 513 513,clip=]{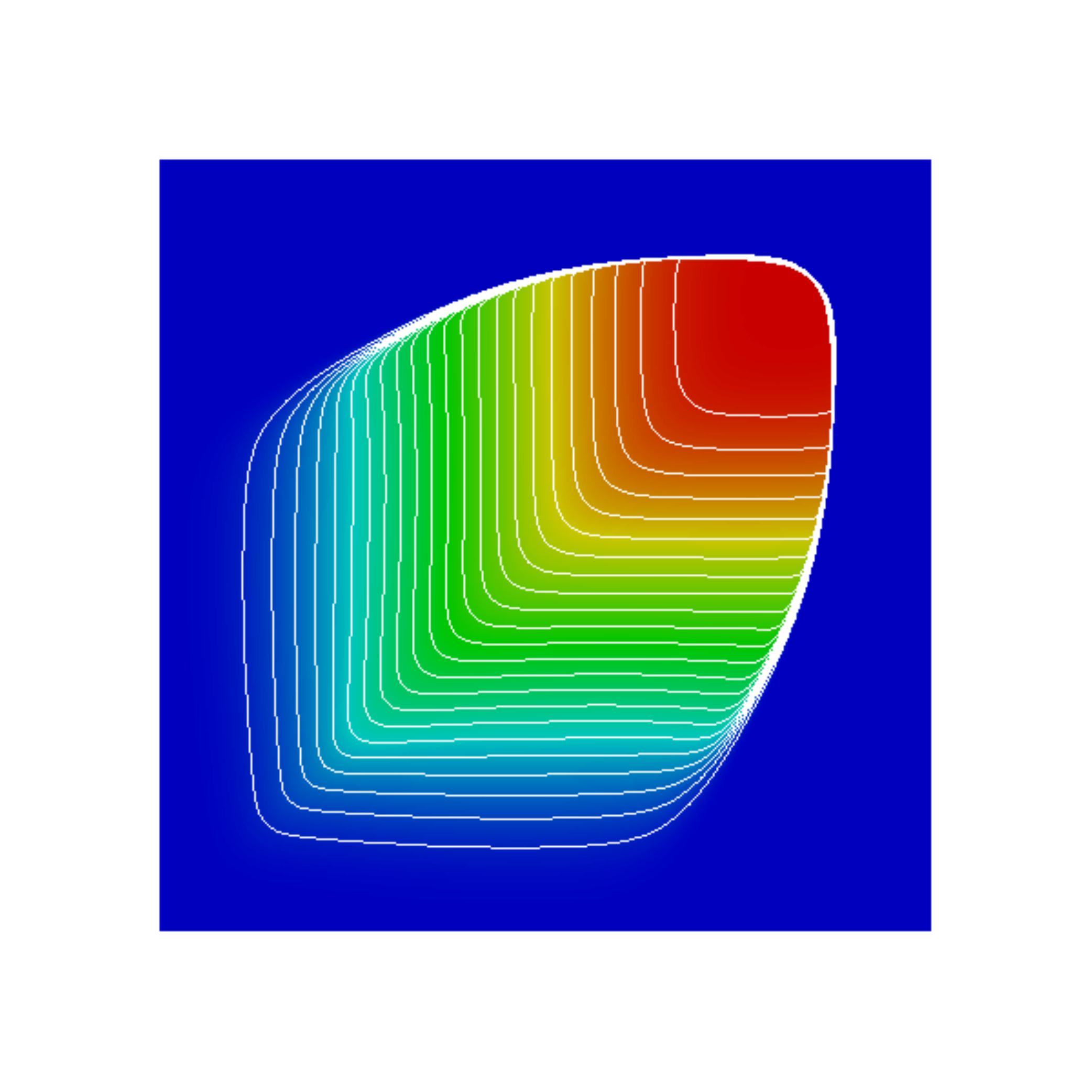}
\includegraphics[width=0.24\textwidth,bb=87 87 513 513,clip=]{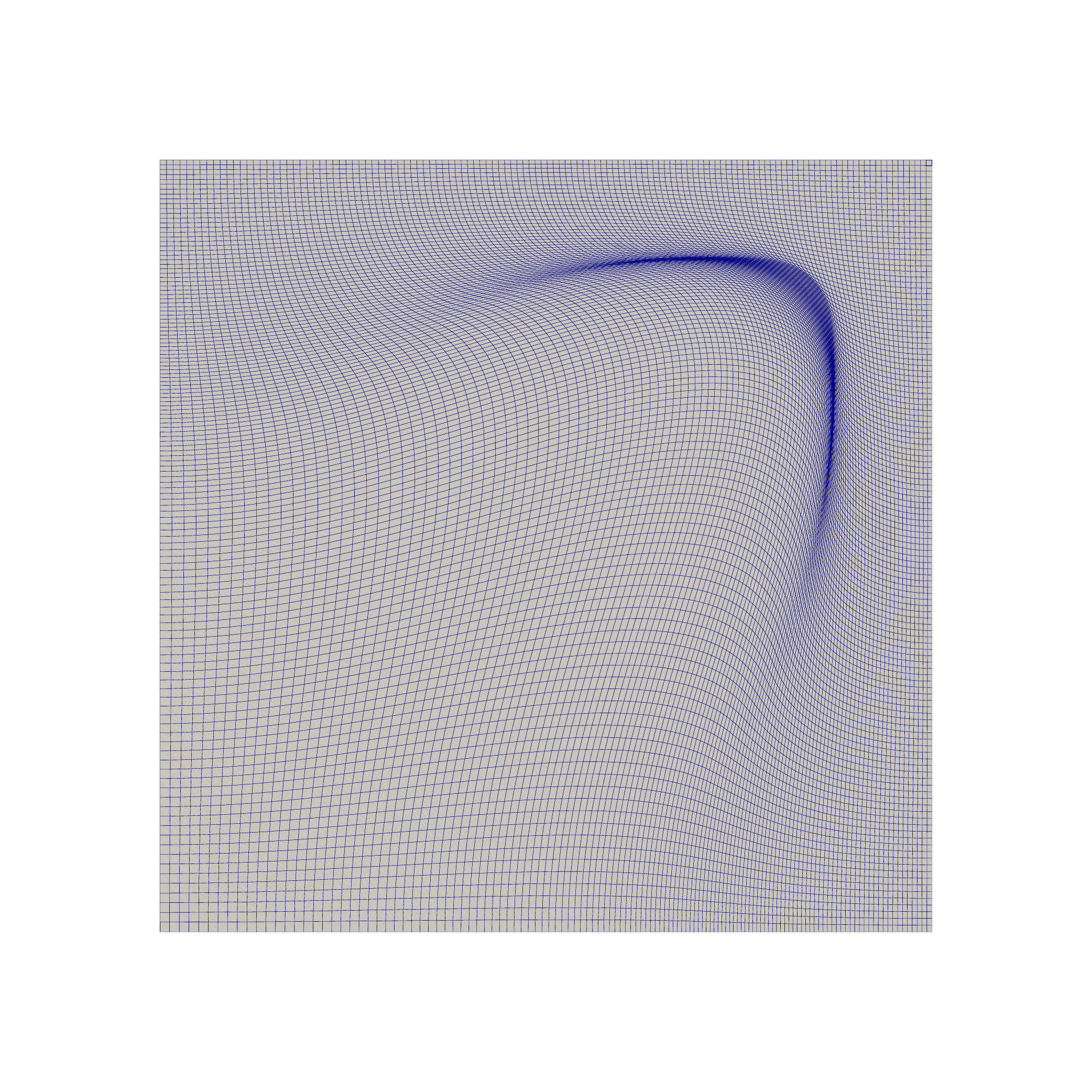}
\includegraphics[width=0.24\textwidth,bb=25 3 396 390,clip=]{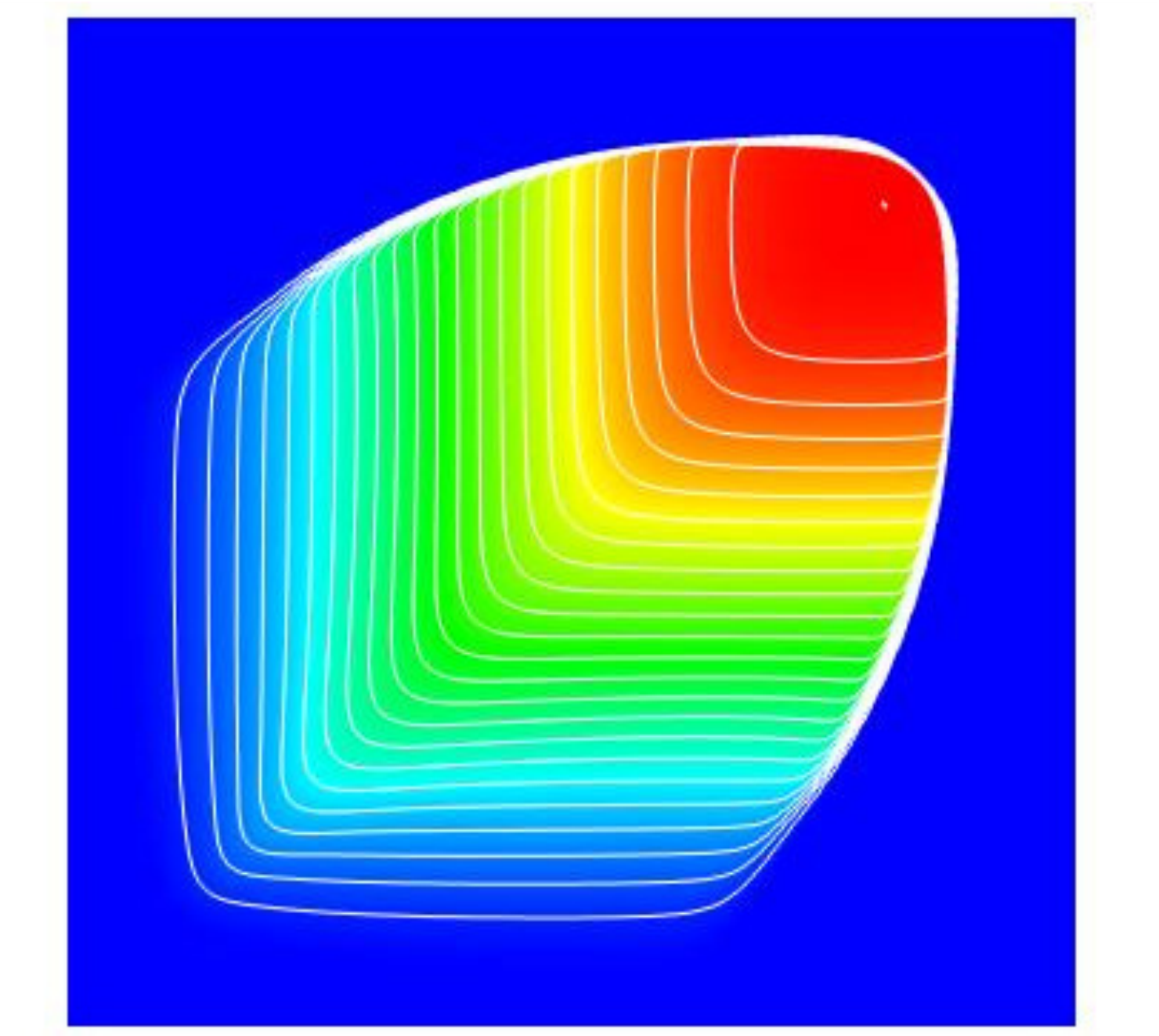}
\includegraphics[width=0.24\textwidth,bb=53 2 384 335,clip=]{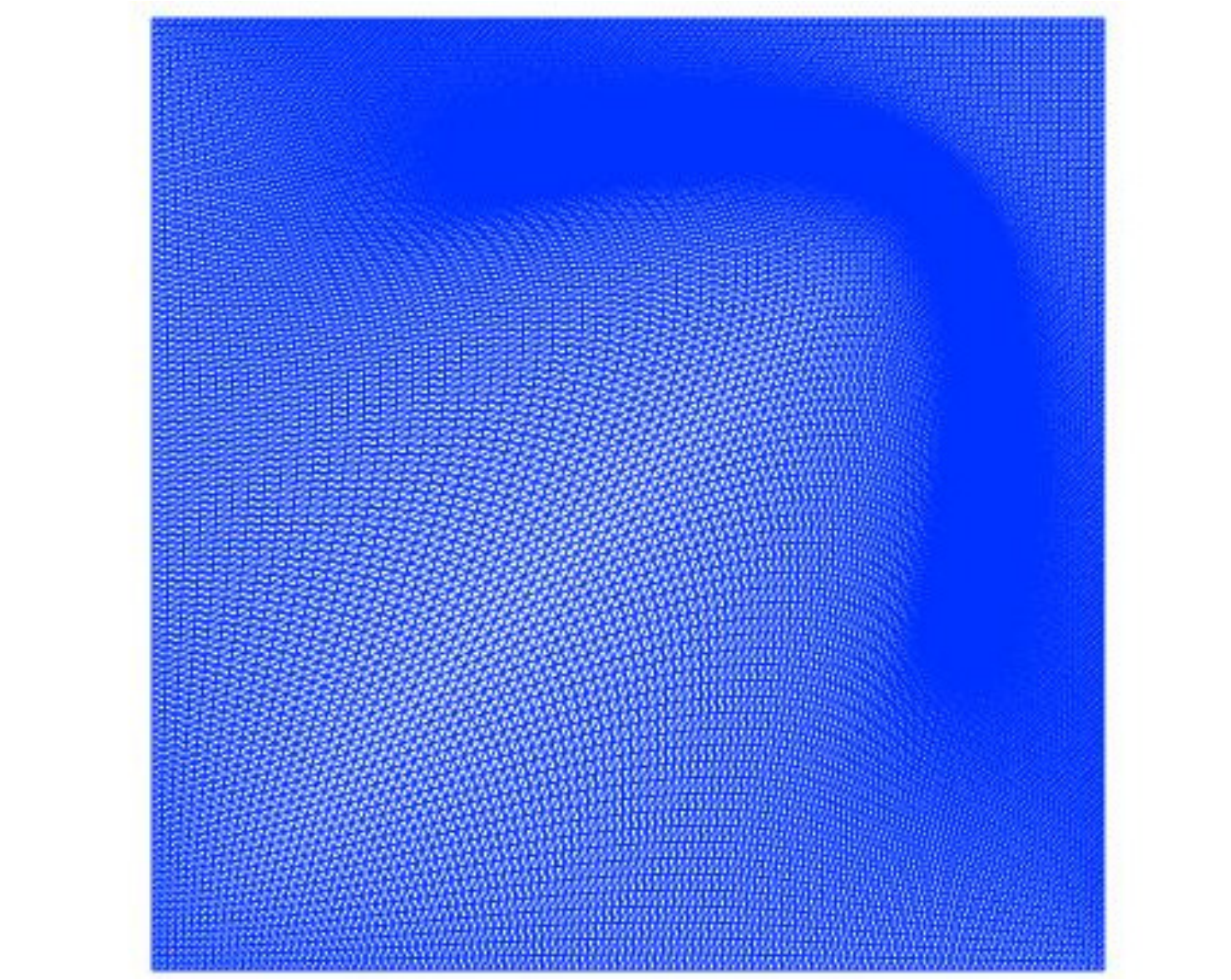}
}
\caption{Burgers equation. Left: $\polQ_1$ FEM with 25 contours; Center left: 
Final $\polQ_1$ mesh; Center right: $\polP_1$ FEM with 25 contours; Right: 
Final $\polP_1$ mesh.}
\label{fig:ale_brugers_new_q1_v1_jlg}
\end{figure}

\subsubsection{Nonconvex flux}
\label{Sec:KPP}
Our last scalar example is a nonlinear scalar conservation law with a non-convex flux 
\begin{equation}
    \label{test_KPP}
     \partial_t u + \DIV \bef(u) = 0,\quad
    u_0(\bx) = 3.25\pi\mathds{1}_{\|\bx\|_{\ell^2}<1}+0.25\pi.
\end{equation}
where $\bef(u)=(\sin u, \cos u)\tr$.  
\begin{figure}[h]
\centering{%
\includegraphics[width=0.24\textwidth,bb=100 74 500 526,clip=]{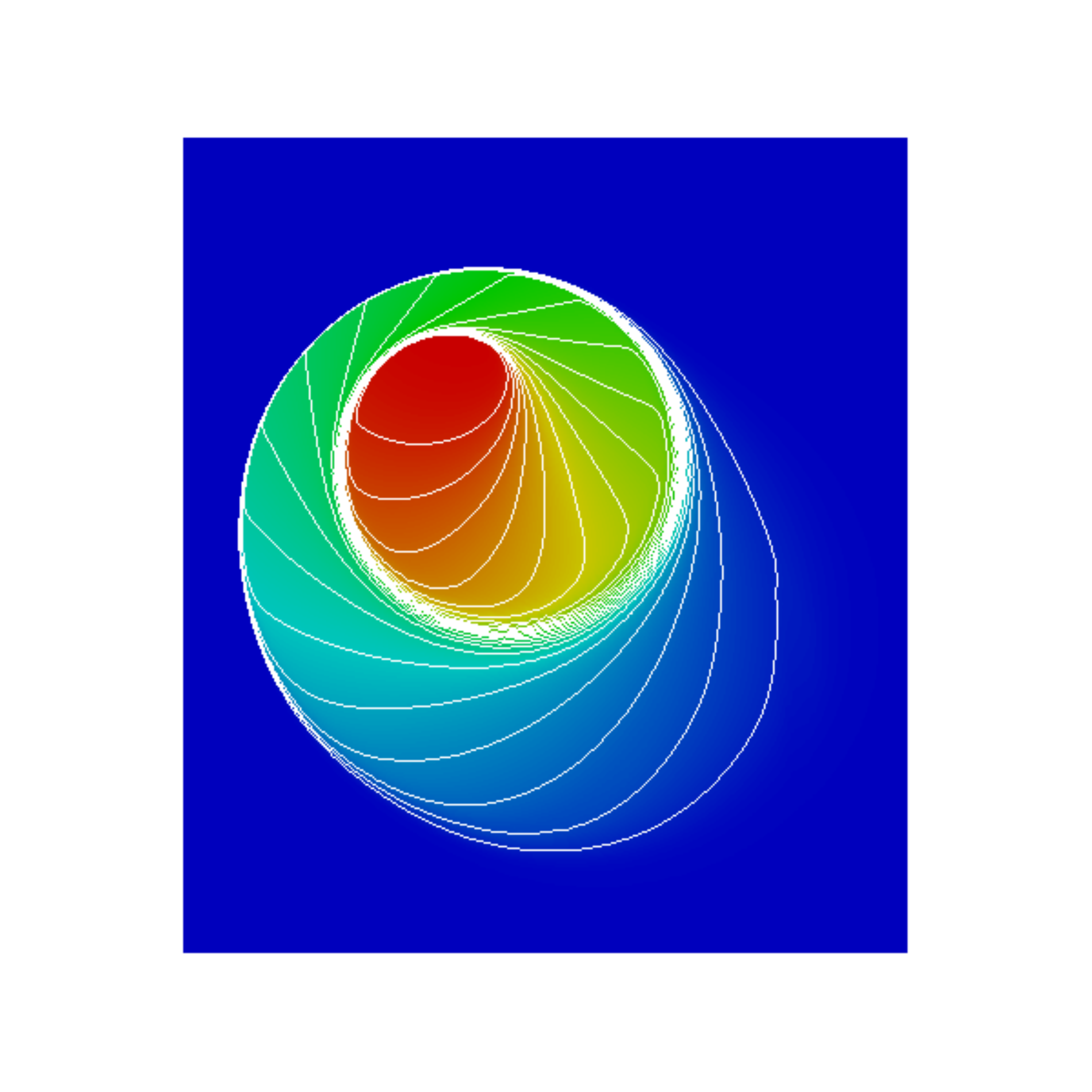}
\includegraphics[width=0.24\textwidth,bb=100 74 500 526,clip=]{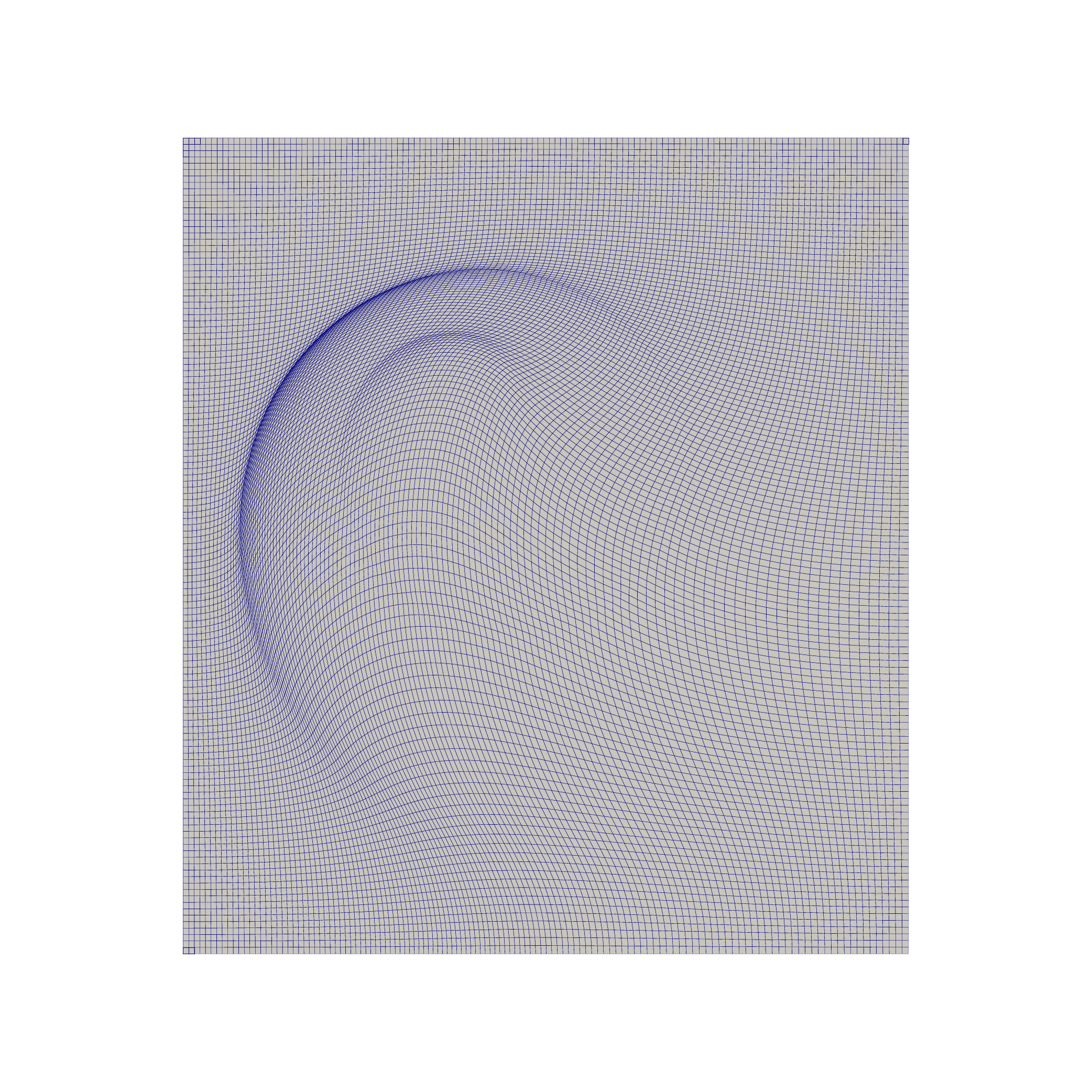}
\includegraphics[width=0.24\textwidth,bb=28 9 385 414,clip=]{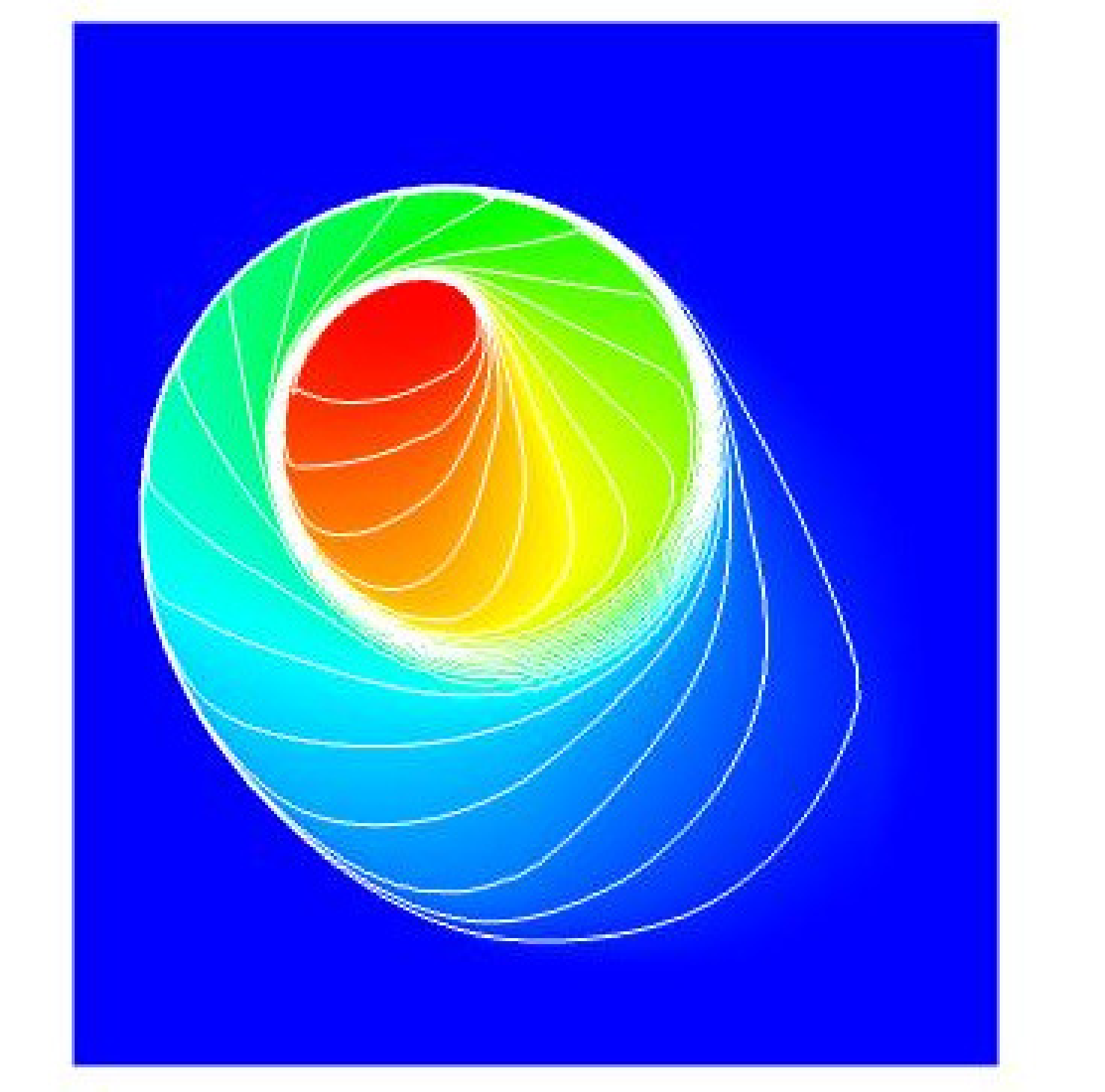}
\includegraphics[width=0.24\textwidth,bb=37 9 386 404,clip=]{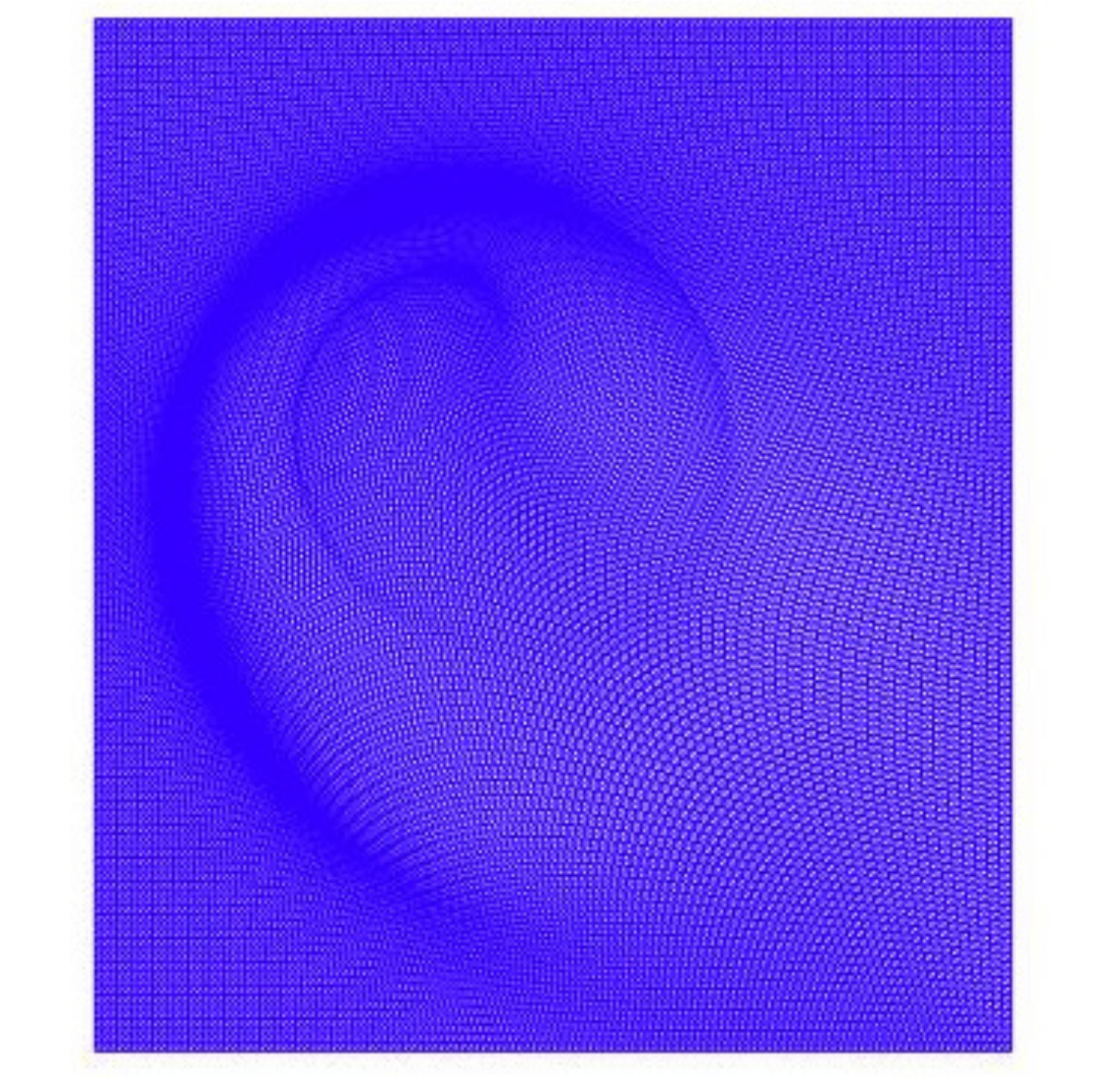}%
}
\caption{KPP problem. Left: $\polQ_1$ FEM with 25 contours; Center left: 
Final $\polQ_1$ mesh; Center right: $\polP_1$ FEM with 25 contours; Right: 
Final $\polP_1$ mesh.}
\label{fig:ale_kpp_jlg}
\end{figure}
This test, henceforth referred
to as KPP, was proposed in \cite{Kurganov_Petrova_Popov_2007}. It is a
challenging test for many high-order numerical schemes because the
solution has a two-dimensional composite wave structure.  The initial
computational domain is $\Dom^0=[-2.5,1.5]\times[-2.0,2.5]$. Note that
the background velocity is constant and equal to
$\bbetaa = (\frac{\sqrt{2}}{2},-\frac{\sqrt{2}}{2})\tr$. It can be
shown that the computational domain keeps a rectangular shape in the
time interval $(0,1)$. The computation is done up to time $T=1$ using
$\polQ_1$ and $\polP_1$ finite elements on structured meshes
$128\times128$ with $CFL=0.1$. The results are shown in
Figure~\ref{fig:ale_kpp_jlg}

\subsection{Compressible Euler equations}
\label{Sec:Euler}
We finish the series of tests by solving the compressible Euler
equations in $\mathbb{R}^2$
\begin{equation}
      \begin{cases}
        \partial_t\rho + \DIV(\rho\bu) &= 0, \\
        \partial_t(\rho \bu) + \DIV(\rho\bu\otimes\bu+p\mathds{I}) &= 0, \\
        \partial_t E + \DIV(\bu(E+p)) &= 0,
      \end{cases}
\end{equation}
 with an ideal gas equation of state, $p=(\gamma-1)(E-\frac12 \rho \|\bu\|_{\ell^2}^2)$ where $\gamma>1$, 
 and appropriate initial and boundary conditions.  The motion of the mesh is done as described in
\eqref{mesh_motion} with
$\ba^{n+1}\idlag = \ba^{n}_i + \dt \bu_h^n(\ba_i^n)$ where $\bu_h^n$
is the approximate fluid velocity at time $t^n$.

\subsubsection{Sod problem}
We use the so-called Sod shocktube problem to
test the convergence of our algorithm (version-1 only); it is a
Riemann problem with the following initial data:
\begin{equation}
      	\rho_0(\bx)= 1.0,\quad
		\bu_0(\bx) = 0.0,\quad
        p_0(\bx) = \mathds{1}_{x_1<0.5}+0.1\,\mathds{1}_{x_1>0.5}.
\end{equation}
and $\gamma=1.4$, see \eg \cite[\S4.3.3]{Toro_2009}. The
computational domain at the initial time is the unit square $(0,1)^2$.
Dirichlet boundary conditions are enforced on the left and right sides of
the domain and we do not enforce any boundary conditions on the upper
and lower sides.  The computation is done up to $T=0.2$. Since no wave
reaches the left and the right boundaries in the time interval
$0<t<T=0.2$, the computational domain remains a square for the whole
duration of the simulation.  The solution being one-dimensional, the
convergence tests are done on five meshes with refinements made only
along the $x_1$-direction.  These meshes have $20\CROSS 4$,
$40\CROSS 4$, $\ldots$, $1280\CROSS 4$ cells and are uniform
at $t=0$. The results of the convergence test are shown in Table
\ref{tab:ale_sod_v1}. We show in this table the $L^1$- and $L^2$-norm
of the error on the density. The convergence orders are compatible
with what is usually obtained in the literature for this problem.
\begin{table}[h]
\begin{center}
\caption{SOD problem, ALE, convergence test, $T=0.2$, $CFL=0.1$}
\label{tab:ale_sod_v1}
\begin{tabular}{|r|c|c|c|c||c|c|c|c|} \hline
 &  \multicolumn{4}{|c||}{$\polQ_1$} & \multicolumn{4}{|c|}{$\polP_1$} \\ \hline
\!\!\# dofs & \multicolumn{2}{|c|}{$L^2$-norm}& \multicolumn{2}{|c||}{$L^1$-norm}
& \multicolumn{2}{|c|}{$L^2$-norm}& \multicolumn{2}{|c|}{$L^1$-norm} \\ \hline
1605 & 2.47E-02 & -    & 1.51E-02 & -      & 2.82E-02&	- &1.77E-02 &	-
 \\ \hline
3205 & 1.84E-02 & 0.43 & 9.99E-03 & 0.60& 2.07E-02	&0.45 &  1.15E-02	 &0.61
 \\ \hline
6405 & 1.36E-02 & 0.42 & 6.42E-03 & 0.64& 1.56E-02	&0.40&  7.45E-03	 &0.63
 \\ \hline
12805& 1.05E-02 & 0.39 & 4.07E-03 & 0.66& 1.26E-02	&0.32 &  4.82E-03	 &0.63
 \\ \hline
\end{tabular}
\end{center}
\end{table}

\subsection{Noh problem}
The last problem that we consider is the so-called Noh problem,
see \eg \cite[\S5]{CaramanaWhalen98}.
The computational domain at the initial time is
$\Dom^0=(-1,1)^2$ and the initial data is
\begin{equation}
      	\rho_0(\bx)= 1.0,\quad
		\bu_0(\bx) = -\frac{\bx}{\|\bx\|_{\ell^2}},\quad
        p_0(\bx) = 10^{-15}.
\end{equation}
A Dirichlet boundary condition is enforced on all the dependent
variables at the boundary of the domain for the entire simulation. We
use $\gamma=\frac{5}{3}$.  The solution to this problem is known;
for instance, the density is equal to
\begin{equation}
\label{true_rho_noh}
	\rho(t,\bx) = 16\mathds{1}_{\{\|\bx\|_{\ell^2}<\frac{t}{3}\}}
+(1+\frac{t}{\|\bx\|_{\ell^2}})\mathds{1}_{\{\|\bx\|_{\ell^2}>\frac{t}{3}\}}.
\end{equation}
The ALE velocity at the boundary of the computational domain is
prescribed to be equal to the fluid velocity, \ie
the boundary moves inwards in the radial direction with speed $1$.
The final time is chosen to be $T=0.6$ in order to avoid that the
shockwave collides with the moving boundary of the computational domain
which happens at $t=\frac34$ since the shock moves radially
outwards with speed $\frac{1}{3}$.

We show in Table \ref{tab:ale_noh_v1} the $L^1$- and the $L^2$-norm of
the error on the density for various meshes which are uniform at
$t=0$: $30\CROSS 30$, $60\CROSS 60$, \etc
\begin{table}[h]
\begin{center}
    \caption{Noh problem, convergence test, $T=0.6$, $CFL=0.2$}
    \label{tab:ale_noh_v1}
\begin{tabular}{|r|c|c|c|c||c|c|c|c|} \hline
 &  \multicolumn{4}{|c||}{$\polQ_1$} & \multicolumn{4}{|c|}{$\polP_1$} \\ \hline
\!\!\!\# dofs & \multicolumn{2}{|c|}{$L^2$-norm}& \multicolumn{2}{|c||}{$L^1$-norm}
& \multicolumn{2}{|c|}{$L^2$-norm}& \multicolumn{2}{|c|}{$L^1$-norm} \\ \hline
961       & 2.60\phantom{E-01} & -    & 1.44\phantom{E-01} & -   & 2.89\phantom{E-01} & -& 1.71 & - \\ \hline
3721      & 1.81\phantom{E-01} & 0.52 & 8.45E-01 & 0.77& 2.21\phantom{E-01} &0.39 & 1.09  &0.64\\ \hline
14641     & 1.16\phantom{E-01} & 0.64 & 4.21E-01 & 1.01& 1.42\phantom{E-01}  & 0.64 &  5.15E-01 &1.08\\ \hline
58081     & 7.66E-01 & 0.60 & 2.10E-01 & 0.99& 9.39E-01  & 0.59 & 2.60E-01  &0.99 \\ \hline
\!231361\!& 5.21E-01 & 0.56 & 1.06E-01 & 0.98&6.33E-01  &0.57&  1.28E-01  &1.02
\\ \hline
\end{tabular}
\end{center}
\end{table}

Preserving the radial symmetry of the solution as best as possible on
non-uniform meshes is an important property for Lagrangian hydrocodes
in the context of the inertial confinement fusion project, which
involves simulating implosion problems, see \cite{CaramanaWhalen98}.
In these problems, mesh-induced violation of the spherical symmetry
may artificially trigger the Rayleigh-Taylor instability and thereby
may hamper the understanding of the real dynamics of the implosion.
We show in the top row of Figure \ref{fig:ale_noh_v1} simulations that
are done on a uniform mesh composed of $96\CROSS 96$ square cells for
the $\polQ_1$ approximation ($(2\CROSS 96)\CROSS (2\CROSS 96)$
triangular cells for the $\polP_1$ approximation), and we compare them
with simulations done on a nonuniform mesh constructed as follows: The
initial square $\Dom^0$ is divided in four quadrants; in the bottom
left quadrant the mesh is composed of $32 \CROSS 32$ square cells; in
the top left quadrant the mesh is composed of $32\CROSS 64$
rectangular cells; in the top right quadrant the mesh is composed of
$64\CROSS 64$ square cells; the bottom right quadrant is composed of
$64\CROSS 32$ rectangular cells. This is a generic test for many
Lagrangian hydrocodes, see \eg \cite[\S8.4]{Dobrev12}. We notice a
slight break of symmetry, but the solution does not develop any
Rayleigh-Taylor-type instability as it is often the case for many
other Lagrangian algorithms.
\begin{figure}[h]
\centerline{%
\includegraphics[width=0.24\textwidth,bb=87 87 513 513,clip=]{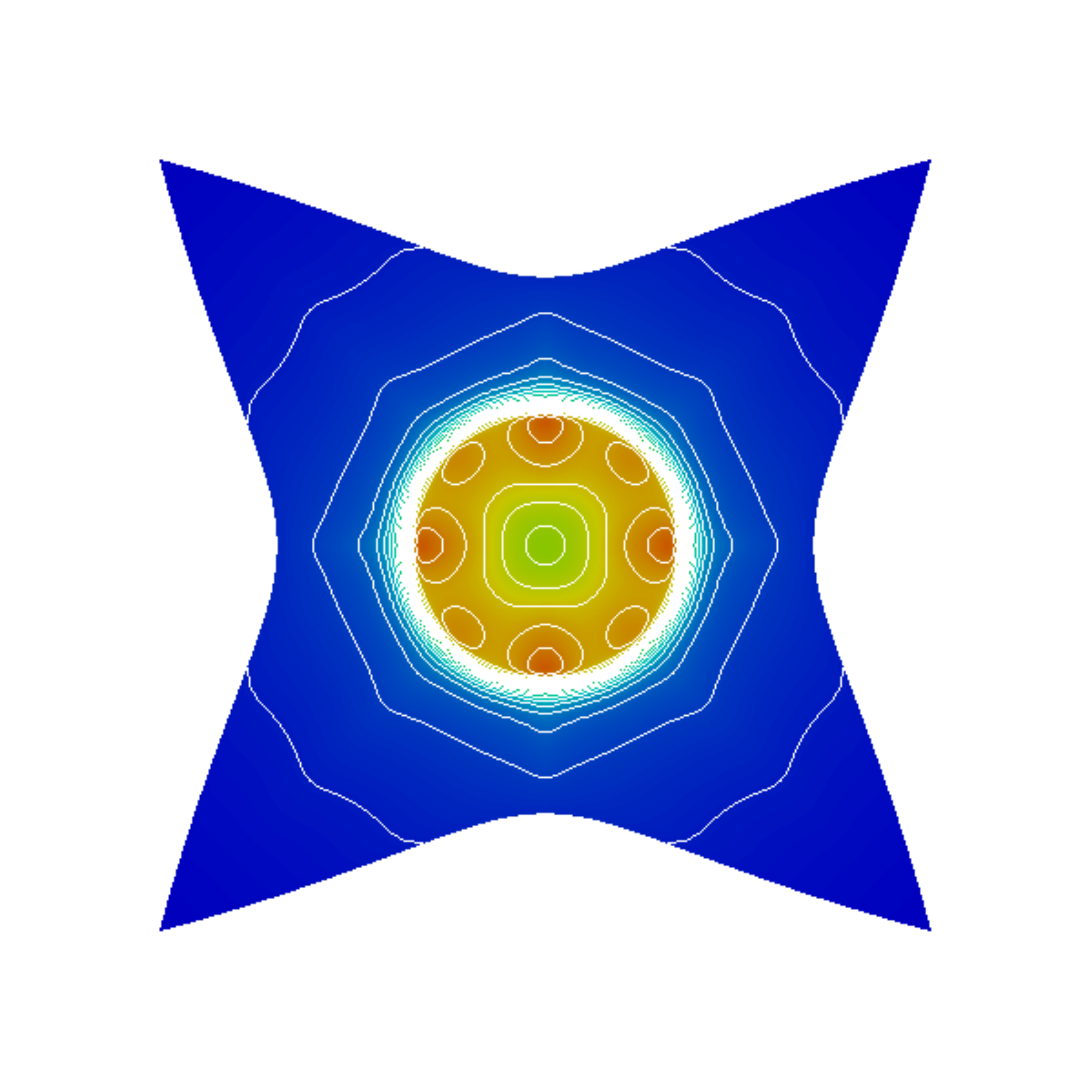}%
\includegraphics[width=0.24\textwidth,bb=87 87 513 513,clip=]{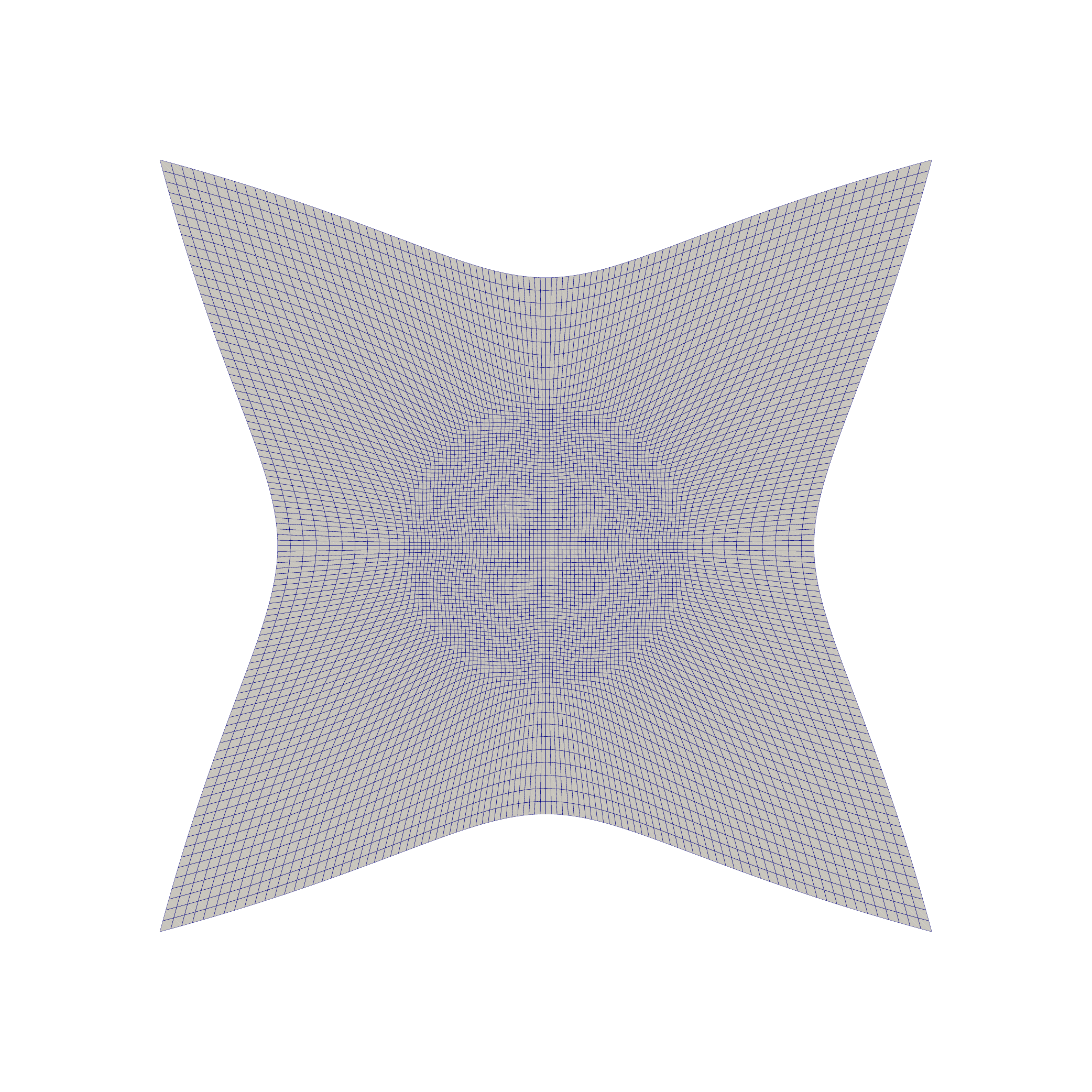}%
\includegraphics[width=0.26\textwidth]{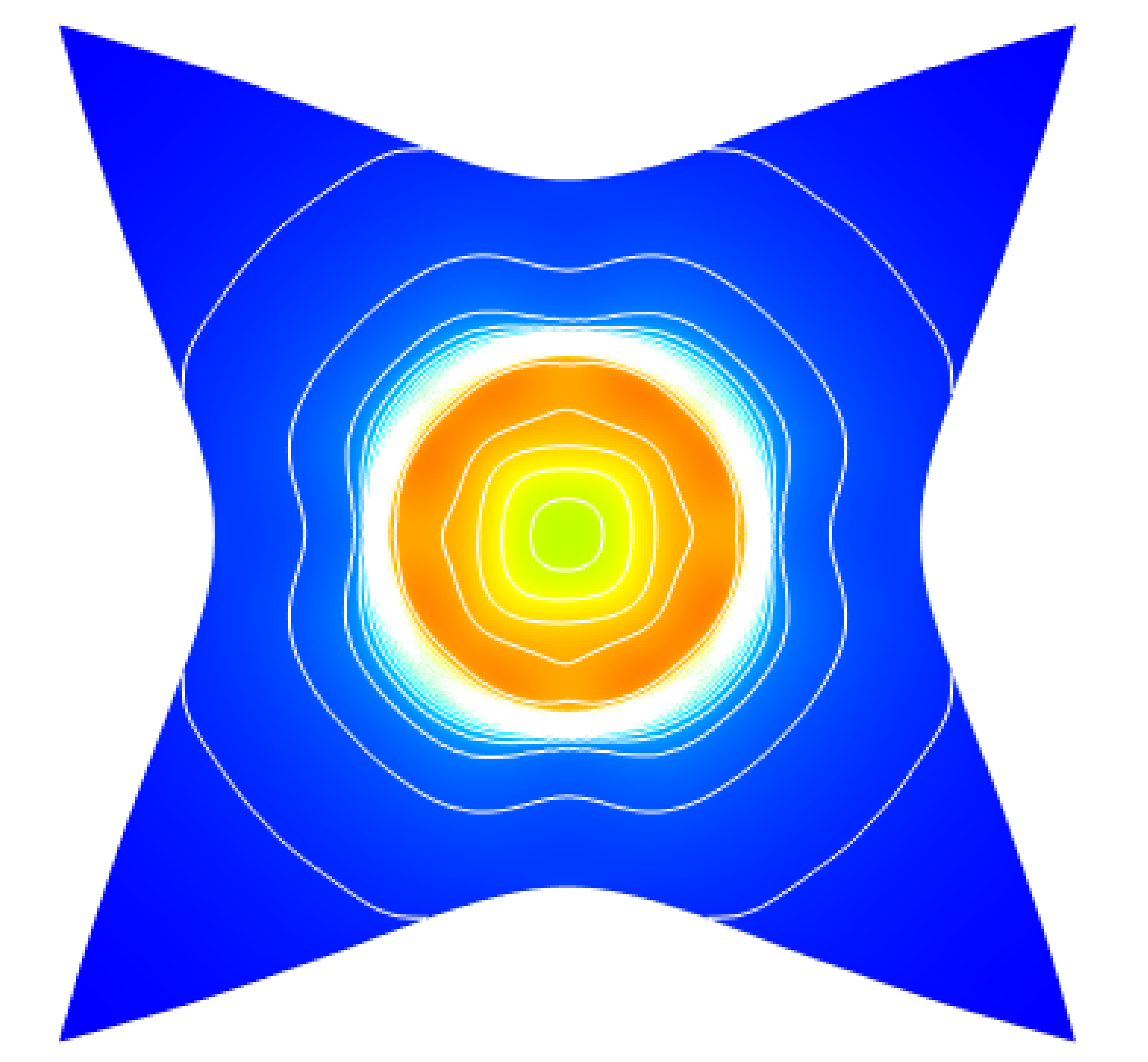}%
\includegraphics[width=0.24\textwidth]{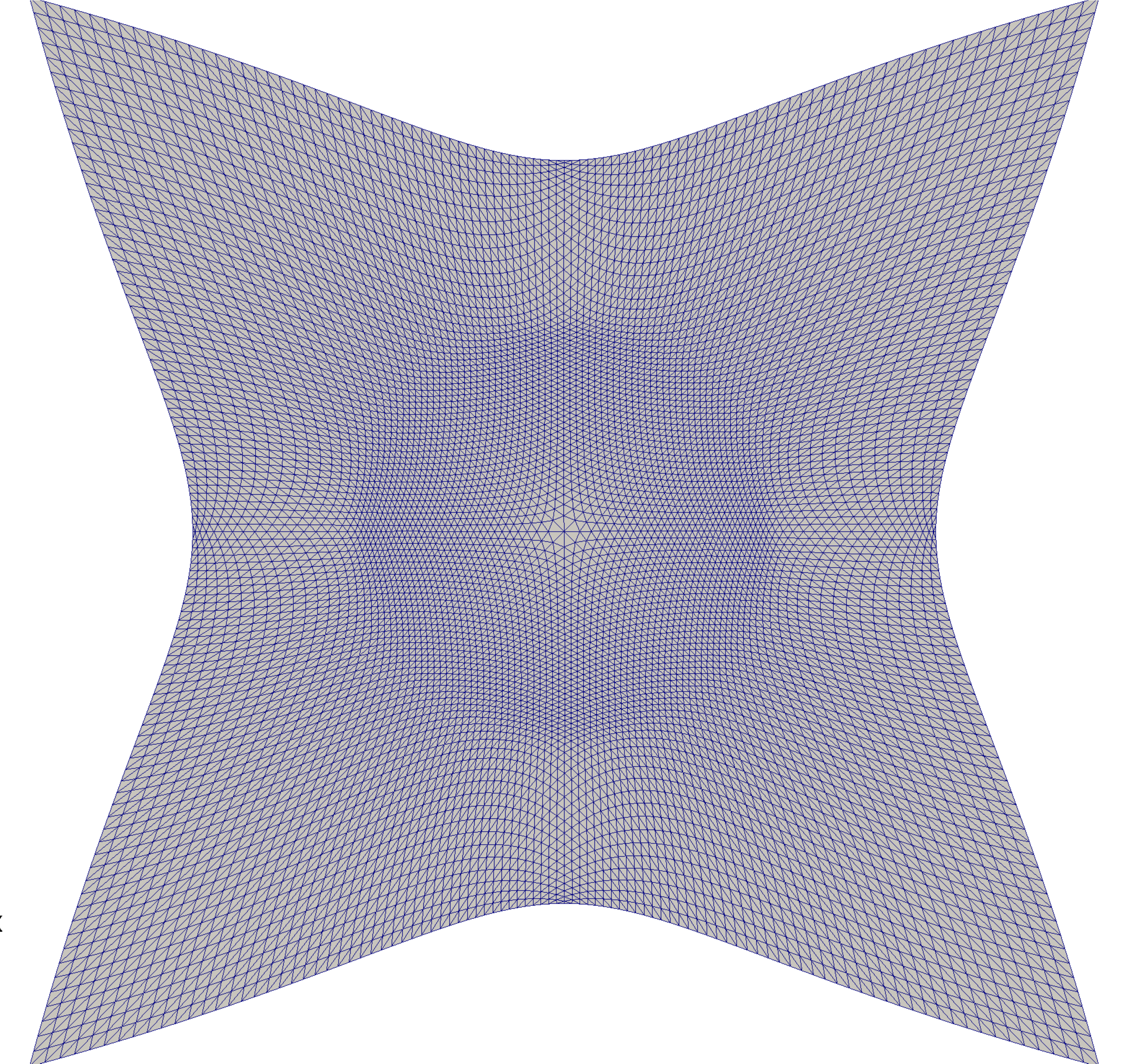}}
\centerline{%
\includegraphics[width=0.24\textwidth,bb=87 87 513 513,clip=]{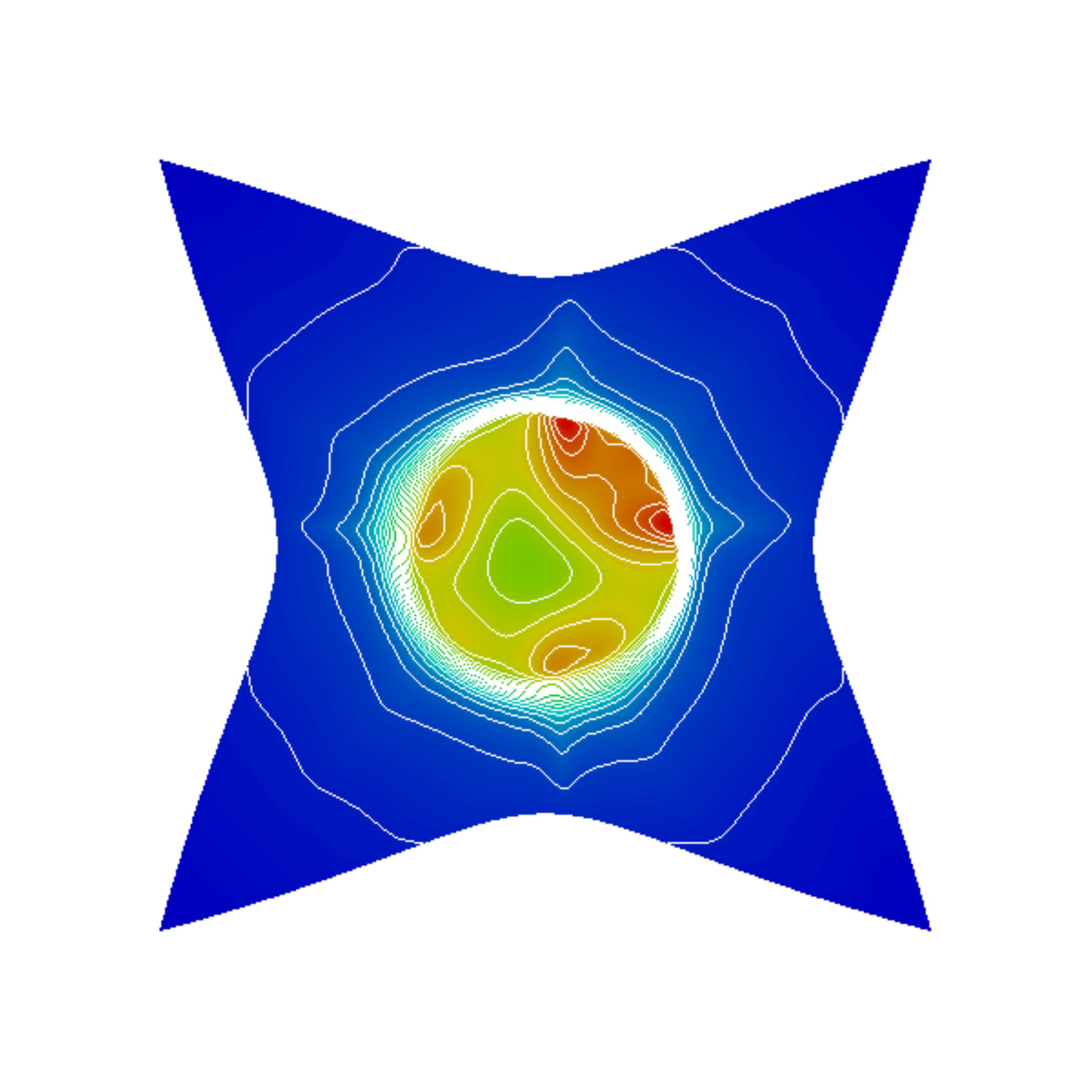}%
\includegraphics[width=0.24\textwidth,bb=87 87 513 513,clip=]{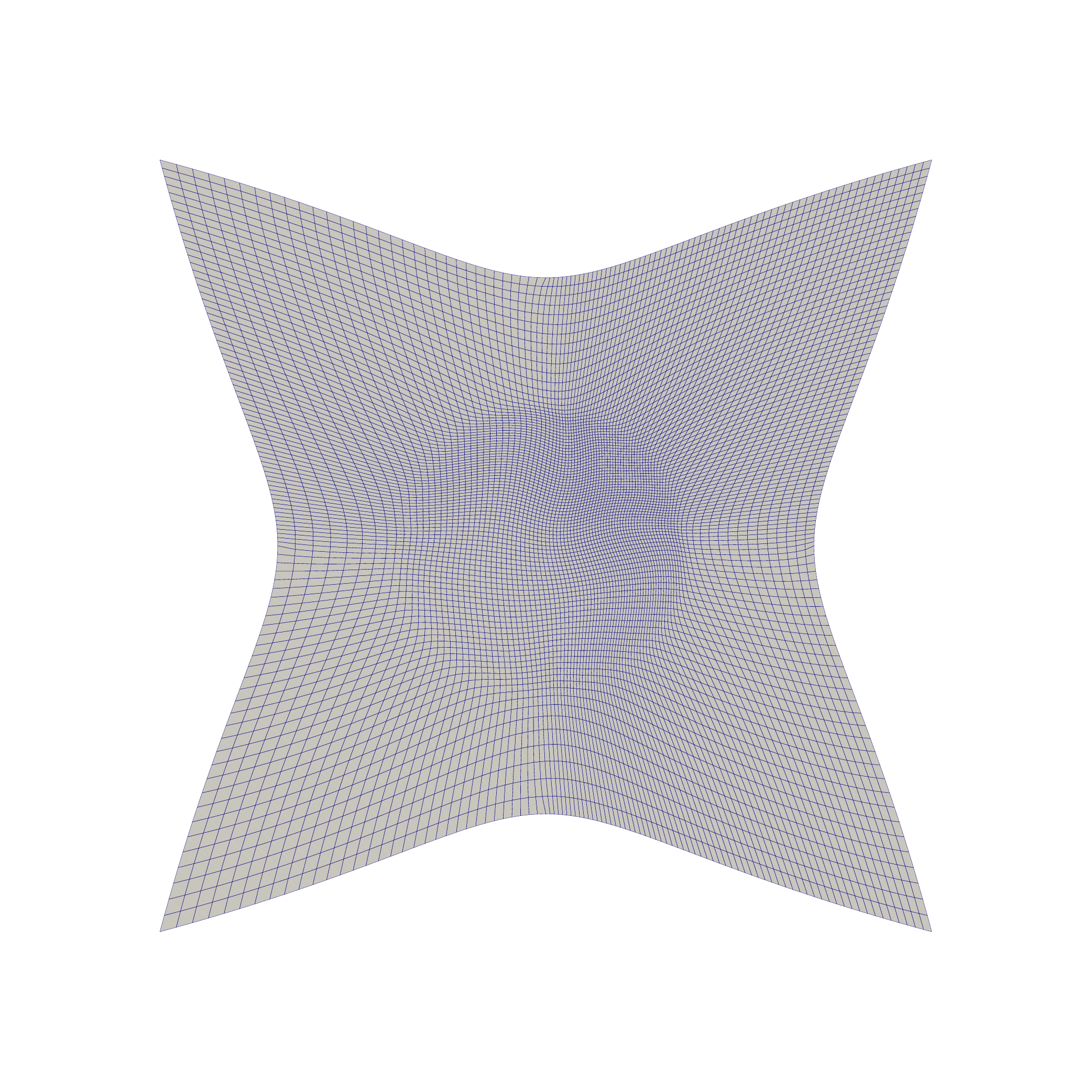}%
\includegraphics[width=0.26\textwidth]{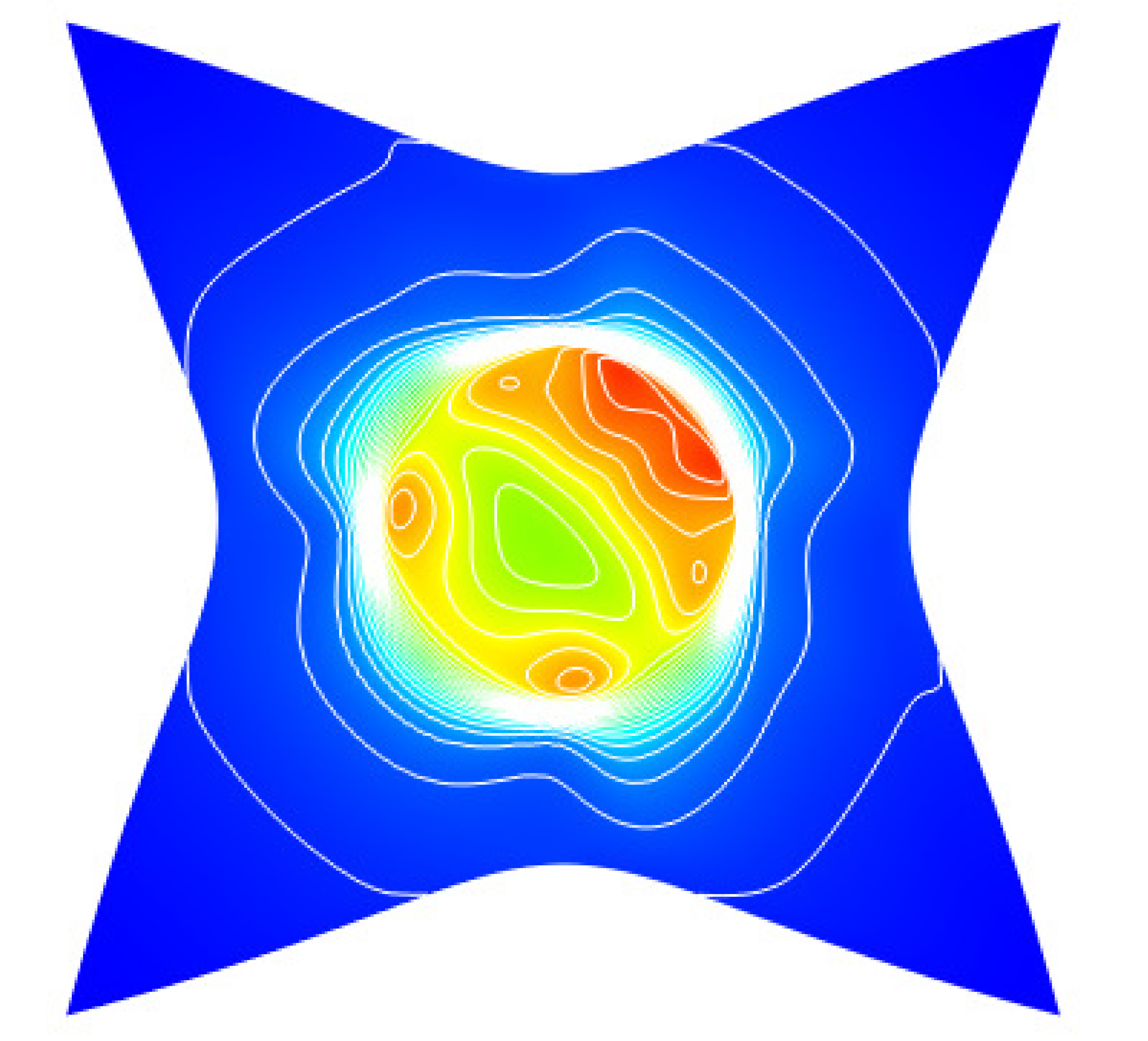}%
\includegraphics[width=0.24\textwidth]{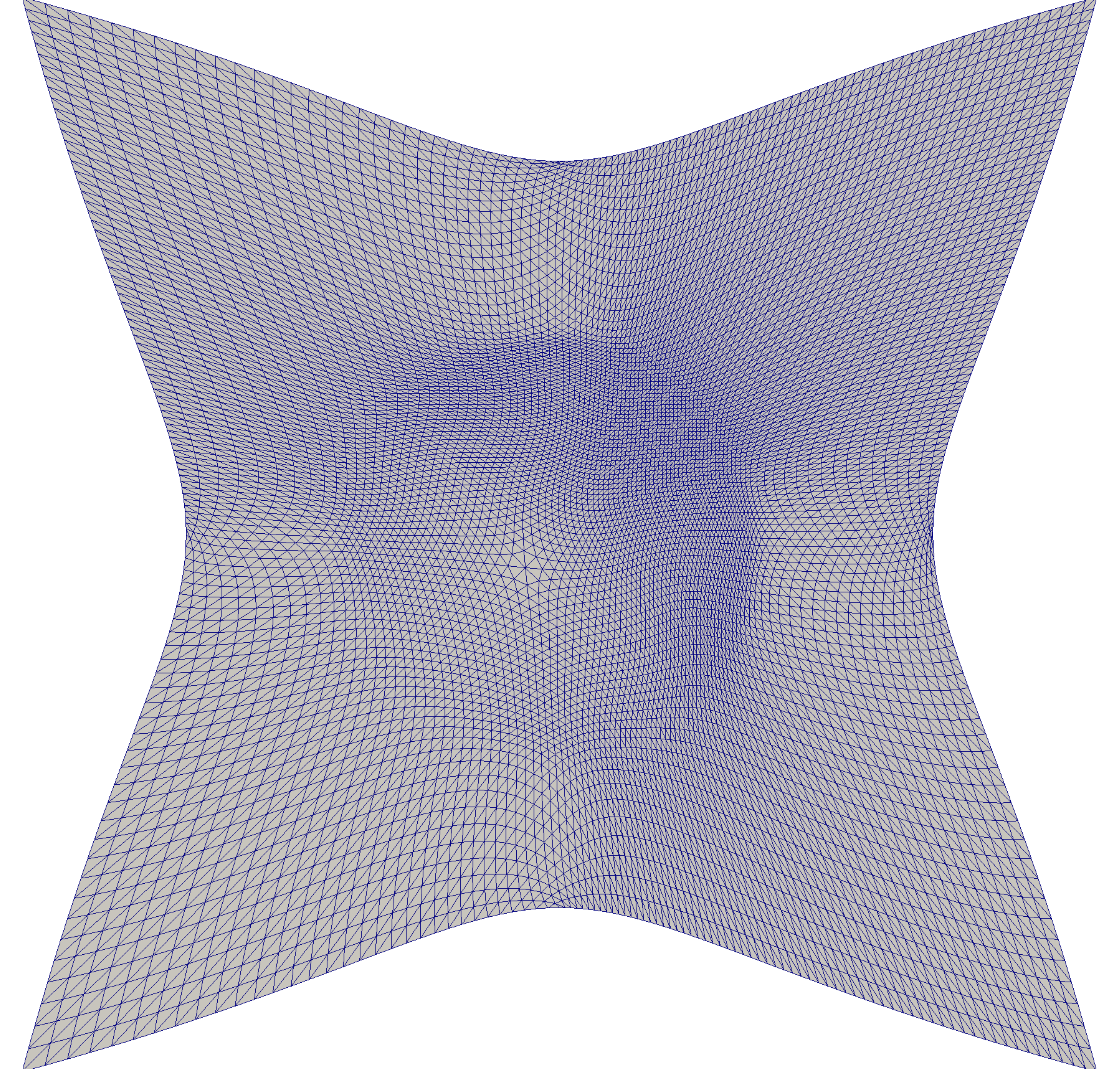}}
\caption{Noh problem. Uniform meshes in top row,
nonuniform meshes in bottom row. From left to right: density field with $\polQ_1$ approximation
(25 contour lines);
mesh with $\polQ_1$ approximation; density field with $\polP_1$ approximation (25 contour lines);
mesh with $\polP_1$ approximation.}
\label{fig:ale_noh_v1}
\end{figure}

We show in Figure~\ref{fig:ale_noh_v1_symmetry} a zoom around the
center of the computational domain for both the $\polQ_1$ and the
$\polP_1$ approximations. We notice a slight motion of the center, but
there is no dramatic breakdown of the structure of the solution.
\begin{figure}[h]
\centering{
{\includegraphics[width=0.24\textwidth]{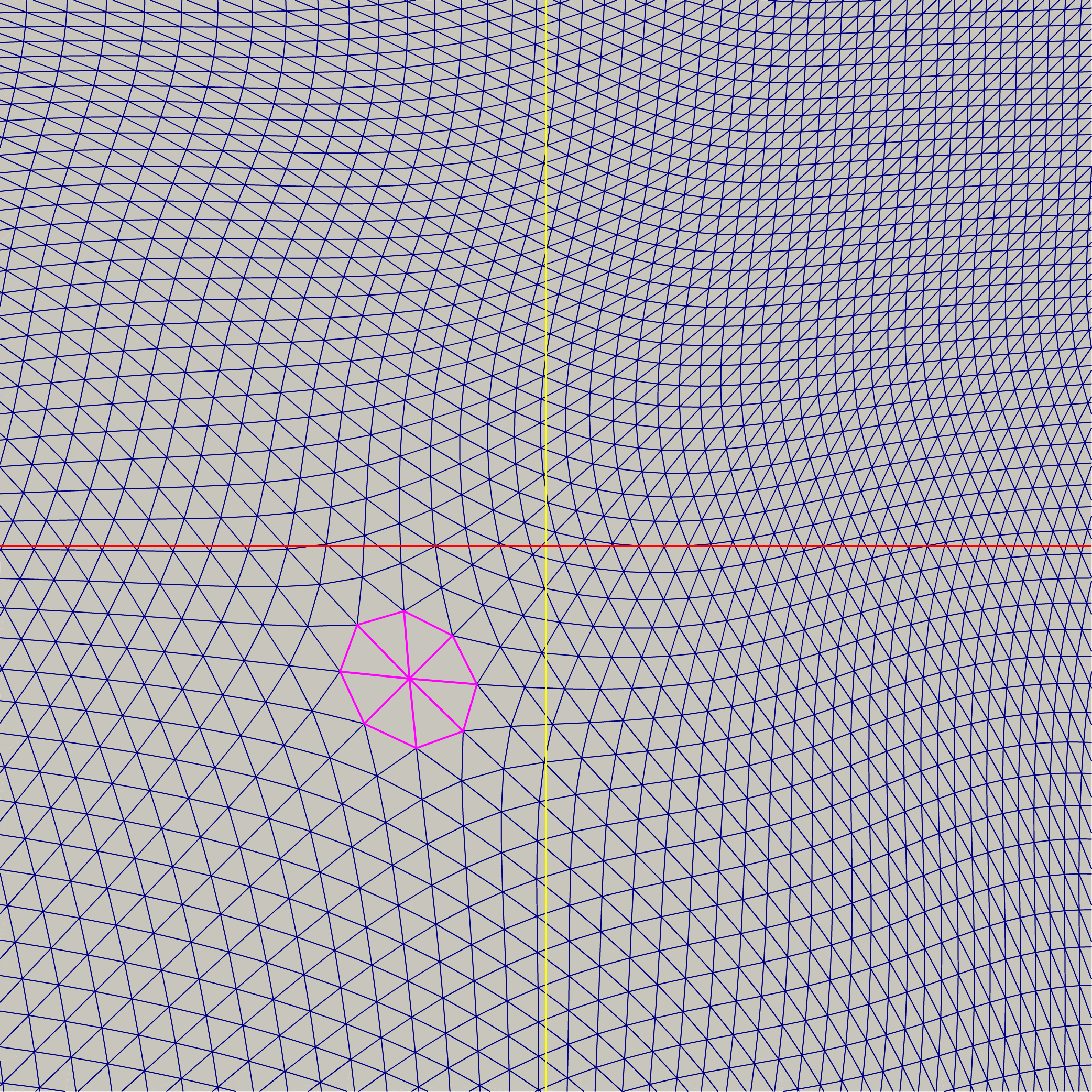}}
{\includegraphics[width=0.24\textwidth]{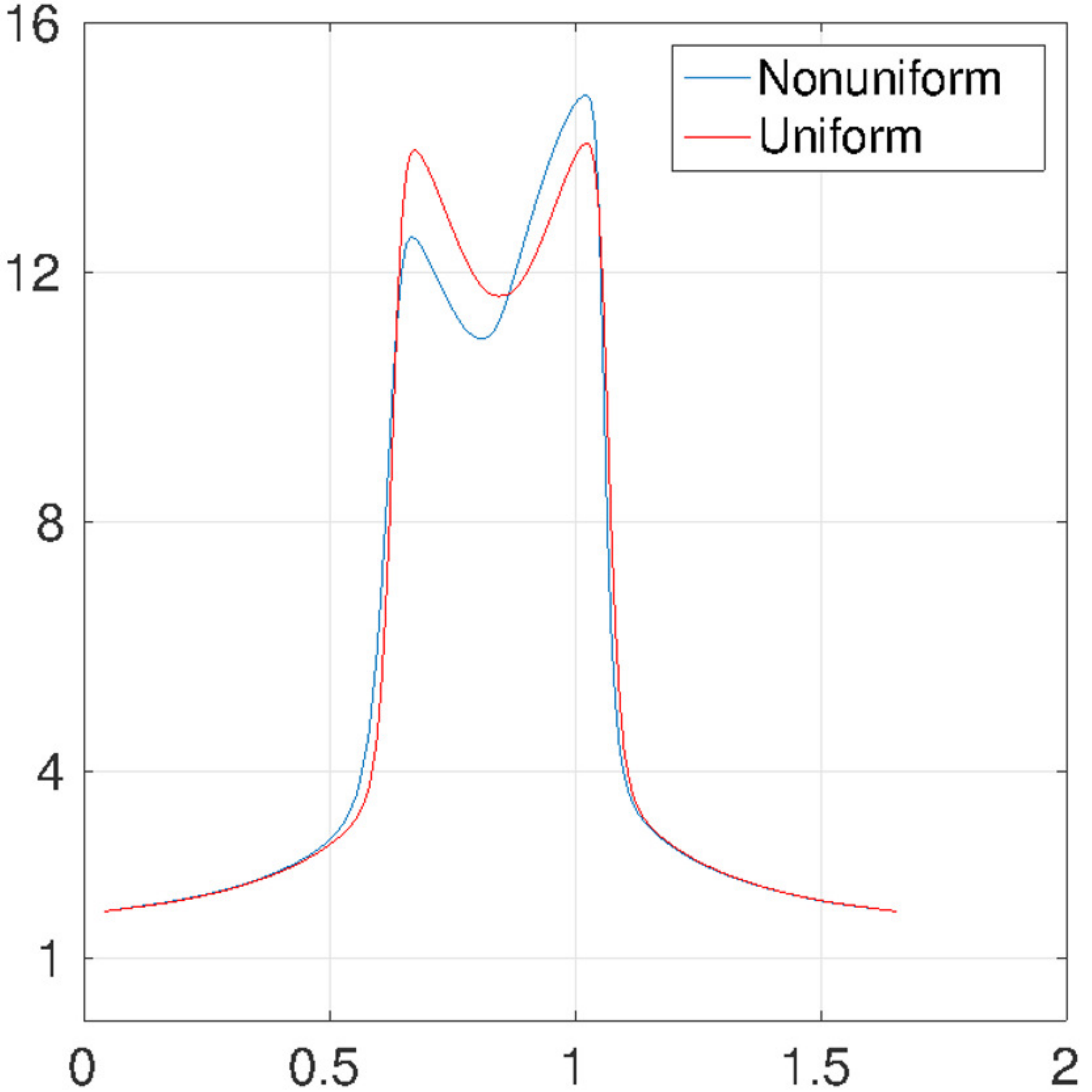}}
{\includegraphics[width=0.24\textwidth]{ALE_MESH_FINAL_LOCAL_32x64-eps-converted-to.pdf}}
{\includegraphics[width=0.24\textwidth]{ALE_SLICE-eps-converted-to.pdf}
}}
\caption{Noh problem. From left to right: Zoom
  around the center of the nonuniform $\polQ_1$ mesh at $T=0.6$ (notice the small
  displacement of the center); Cross section along the line connecting
  the points $(-1,-1)$ and $(1,1)$ for the $\polQ_1$ solutions on
  the uniform mesh and on the nonuniform mesh; 
Zoom around the center of the nonuniform $\polP_1$ mesh at $T=0.6$; 
Cross section along the line connecting
  the points $(-1,-1)$ and $(1,1)$ for the $\polP_1$ solutions on
  the uniform mesh and on the nonuniform mesh; }
\label{fig:ale_noh_v1_symmetry}
\end{figure}

\section{Concluding remarks} In this paper we have developed a
framework for constructing ALE algorithms using continuous finite
elements. The method is invariant domain preserving on any mesh in
arbitrary space dimension. The methodology applies to any hyperbolic
system which has such intrinsic property.  If the system at hand has
an entropy pair, then the method also satisfies a discrete entropy
inequality.  The time accuracy of one of the methods (scheme~1) can be
increased by using SSP time discretization techniques.  This makes the
method appropriate to use as a safeguard when constructing high-order
accurate discretization of the system which may violate the invariant
domain property. The equivalence between the conservative and
non-conservative formulations implies the DGCL condition (preservation
of constant states). The new methods have been tested on a series of
benchmark problems and the observed convergence orders and numerical
performance are compatible with what is reported in the literature.

\bibliographystyle{abbrvnat} 
\bibliography{ref}
\end{document}